\documentclass{IEEEtran} 
\IEEEoverridecommandlockouts                            
\usepackage[margin=0cm]{caption}
\usepackage{epsf}
\usepackage{epsfig}
\usepackage{times}
\usepackage{amssymb}
\usepackage{amsmath}
\usepackage{amsfonts}
\usepackage{latexsym}
\usepackage{url}
\usepackage{mathrsfs}
\usepackage{algorithm}
\usepackage{caption}
\usepackage{datetime}
\usepackage{bm}
\usepackage{algorithm,algorithmic}
\usepackage{subfigure}
\usepackage{framed} 
\usepackage[framed]{ntheorem}
\usepackage{cite}
\usepackage{lipsum}
\usepackage{amstext}  
\usepackage{array}      
\usepackage{mathabx}
\usepackage{color}
\allowdisplaybreaks[2]                   
\newframedtheorem{frm-definition}{Definition}
\newtheorem{theorem}{Theorem}
\newtheorem{definition}{Definition}
\newtheorem{lemma}{Lemma}
\newtheorem{corollary}{Corollary}

\newtheorem{remark}{Remark}

\newcommand{\hN}{\mathcal{N}}

\newcommand{\diag}{{\textrm{diag}}}

\newcommand{\cN}{{\cal N}}

\newcommand{\cT}{{\cal T}}

\newcommand{\cE}{{\cal E}}

\newcommand{\cK}{{\cal K}}

\newcommand{\cM}{{\cal M}}

\begin{document}
\title{\LARGE Gradient-Based Multi-Area Distribution System State Estimation}
\author{Xinyang Zhou, Zhiyuan Liu, Yi Guo, Changhong Zhao, Jianqiao Huang, and Lijun Chen
\thanks{X. Zhou is with Power System Engineering Department, National Renewable Energy Laboratory, Golde, CO 80401, USA (email: xinyang.zhou@nrel.gov).}
\thanks{Z. Liu and L. Chen are with Department of Computer Science, University of Colorado, Boulder, CO 80309, USA (emails: \{zhiyuan.liu, \mbox{lijun.chen}\}@colorado.edu).}
\thanks{Y. Guo is with the Department of Mechanical Engineering, The University of Texas at Dallas, Richardson, TX, USA (email:yi.guo2@utdallas.edu).}
\thanks{C. Zhao is with the Department of Information Engineering, The Chinese University of Hong Kong, HK (email: chzhao@ie.cuhk.edu.hk).}
\thanks{J. Huang is with the Department of Electrical and Computer Engineering, Illinois Institute of Technology, Chicago, USA (email: jhuang54@hawk.iit.edu).}
}

\maketitle

\begin{abstract}
The increasing distributed and renewable energy resources and controllable devices in distribution systems make fast distribution system state estimation (DSSE) crucial in system monitoring and control. We consider a large multi-phase distribution system and formulate DSSE as a weighted least squares (WLS) problem. We divide the large distribution system into smaller areas of subtree structure, and by jointly exploring the linearized power flow model and the network topology, we propose a gradient-based multi-area algorithm to exactly and efficiently solve the WLS problem. The proposed algorithm enables distributed and parallel computation of the state estimation problem without compromising any performance. Numerical results on a 4,521-node test feeder show that the designed algorithm features fast convergence and accurate estimation results. Comparison with traditional Gauss-Newton method shows that the proposed method has much better performance in distribution systems with a limited amount of reliable measurement. The real-time implementation of the algorithm tracks time-varying system states with high accuracy. 
\end{abstract}



\section{Introduction}
Conventional distribution networks are passive systems with predictable loads and generations, and they thus feature with relatively stable system states. In the past decade, however, the penetration level of distributed energy resources (DERs)---a large portion of which are renewable energy resources such as roof-top photovoltaic panels---has been increasing in the distribution systems. The intermittent power injections from the DERs result in rapid changes in system states and impose challenges on distribution system operator (DSO) to monitor and control the system in real time.

State estimation uses limited measurement data to calculate the most likely values of true system states that can determine all other system parameters. Transmission systems state estimation (TSSE) is a well-explored area \cite{gomez2004power}. Traditionally, TSSE is formulated as a weighted least squares (WLS) problem and solved through Gauss-Newton method. However, applying a similar solution method to distribution system state estimation (DSSE) is not straightforward because of the following reasons \cite{della2014electrical,dehghanpour2019survey}:

\noindent 1) Distribution networks may contain thousands or tens of thousands of nodes. Directly implementing the traditional Gauss-Newton methods  for large distribution systems cannot provide fast estimation due to its scalability issues, limiting its application in fast real-time state estimation.

\noindent 2) Distribution networks have high line resistance. Therefore, model simplification for TSSE such as ignoring line resistance \cite{gomez2004power}, becomes inaccurate for DSSE. 

\noindent 3) Unlike transmission systems deployed with redundant reliable measurement devices, usually a limited number of measurement devices are deployed in distribution systems, potentially leading to insufficient observability \cite{bhela2018enhancing}. Even though the observability issue can be mitigated by using pseudo-measurement \cite{primadianto2017review}, the resultant ill-conditioned WLS problem may cause Gauss-Newton method to perform poorly, especially in large systems.

\noindent 4) Different from balanced transmission systems, distribution systems are usually unbalanced and thus requires multi-phase DSSE formulation and algorithms.

\noindent 5) The stability of Gauss-Newton method is known to be sensitive to the initial point and needs additional algorithm design to find good ones \cite{yao2018distribution,zamzam2019data}. Reference \cite{yao2018distribution} also gives examples of Gauss-Newton not converging with a flat start.

These issues motivate us to consider different state estimation methodologies that can better accommodate future distribution systems. In this work, we consider a large radial multi-phase distribution network where voltage magnitudes are volatile because of high penetration level of DERs, and they need to be monitored closely. We choose real and reactive nodal power injections of all nodes as the system states to further estimate voltage magnitudes. Based on limited measuring devices---including (pseudo-)measurement of all load nodes along with a set of nodes with voltage magnitude measurement---we formulate a WLS problem to estimate the system states. Because this WLS problem may be ill-conditioned with huge difference among weights for different measured values, traditional Gauss-Newton method can resultantly generate inaccurate estimation results. Therefore, we propose to solve it with more robust gradient algorithms. Next, we jointly explore the linearized distribution flow (LinDistFlow) model and the system topology, and we equivalently implement the gradient algorithm in a hierarchical manner. Specifically, we divide the distribution network into multiple areas featuring subtree topology, and we assign an area monitoring system (AMS) to each area. We design a multi-area algorithm such that the AMS and DSO communicate and collaborate to exactly solve the original WLS problem. In this way, the data measurement and computation of the original large WLS problem is divided among AMS and DSO, enabling distributed and parallel computation and fast convergence. 

We test the proposed algorithm on the IEEE 37-node system and a three-phase unbalanced 4,521-node test feeder that is based on IEEE 8500-node system and EPRI Ckt7 test system. Simulation results show that the proposed algorithm has much faster convergence and much higher estimation accuracy  than Gauss-Newton method in distribution systems with limited reliable measurement. We then implement the design in a time-varying scenario with one gradient step update at each second according to the temporal granularity of loads and PV generation data. The real-time multi-area DSSE show accurate tracking of the true system states.

\subsection{Related Works}
Traditionally, Gauss-Newton method is applied to solve DSSE \cite{lu1995distribution,wang2004revised}. However, there exist problems adapting this method from transmission systems to distribution systems.  As mentioned, unlike transmission systems with redundant measurement, distribution systems usually have very limited measuring devices. To ensure observability, pseudo-measurement of load nodes power injections with large errors are usually applied. Such big errors may lead to ill-conditioned WLS problem formulation and undesired nonconvergence and inaccuracy results by Gauss-Newton method. Recent work \cite{zamzam2019data} observes non-convergence in DSSE with Gauss-Newton method improves the convergence by finding ``warm starts" through data-driven method. In Section~\ref{sec:compare} we will also report unsatisfying results generated by Gauss-Newton method under realistic scenarios where reliable measurement devices are scarce.

Another problem associated with Gauss-Newton method is its lack of scalability in large systems. As is known, the computational complexity of Gauss-Newton method is between $\mathcal{O}(N^2)$ to $\mathcal{O}(N^3)$ for systems with $N$ nodes, and it significantly increases for large systems.
A promising line of work for efficiently solving large systems state estimation is based on multi-area state estimation (MASE). MASE with Gauss-Newton method has been widely applied for TSSE where it divides large transmission systems into several sub-areas and approximately solves the WLS problem locally \cite{gomez2011multilevel,zhao2005multi,kekatos2012distributed}. For distribution systems, on the other hand, there are fewer studies \cite{primadianto2017review}. References \cite{muscas2015multiarea,pau2017efficient} divide the distribution system based on geographical and topological constraints as well as measurement availability and account correlation among divided areas to improve estimation accuracy. Reference \cite{gomez2012state} divides distribution networks based on feeders and substations. However, these works can only approximately solve the original WLS problem. 
\cite{zhu2014power} designs distributed SDP solver to exactly solve the WLS based on alternative direction method of multipliers (ADMM) but its performance relies on strong assumptions. Equally importantly, these existing multi-area-based DSSE works lack large system tests to illustrate the scalability of their designs. 

Other recent works on DSSE solution methods include references \cite{donti2018matrix,genes2018robust}, which use matrix completion method---a method for estimating missing values in low-rank matrices---to handle low observability issues in distribution systems. \cite{zhao2019power}, as well as the references therein, applies dynamic state estimation based on Kalman filter to monitor power systems. These solution methods, however, may not be  scalable enough to handle large systems efficiently, thus prohibiting real-time implementation for fast-changing distribution system states. 

\subsection{Contributions}
Although most existing works on DSSE are based on Gauss-Newton method---see two recent surveys \cite{primadianto2017review,dehghanpour2019survey}---there is little work on the gradient-based method. Aiming to provide efficient and accurate DSSE for large distribution systems with limited reliable measurements, our proposed work makes significant contributions in the following aspects: 
\begin{enumerate}
\item We propose a robust gradient-based algorithm for solving WLS with potentially large errors associated with pseudo-measurement, and we provide rigorous analytic performance characterization.
\item We design a novel multi-area implementation of the gradient-based algorithm by exploring the tree/subtree topological structure of distribution systems. Such a design enables parallel and distributed computation for large DSSE problems without losing performance compared with the centralized implementation.
\item We compare the numerical performance between the proposed gradient-based method and Gauss-Newton method on IEEE 37-node test feeder and a 4,521-node multi-phase distribution system based on IEEE 8500-node test system and EPRI Ckt7 test system to show that gradient-based method is more scalable in large system and that under the situations of scarce reliable measurement, gradient-based method consistently generates more robust and more accurate estimation results.
\item We further extend the implementation of the proposed method to real-time scenario with load and PV generation data changing every one second and illustrate its accurate online state estimation performance.
\end{enumerate}

\begin{remark}
To incorporate other types of measurement, as well as related problems such as optimal measurement placement, bad data detection, network topology estimation etc., is crucial for better state estimation results. Our ongoing efforts are exploring how to efficiently and accurately involve other types of measurement including voltage phases, branch power flows, and so on into the proposed multi-area state estimation framework. However, despite that we only use voltage magnitudes measurement and (pseudo-)measurement of load nodes, the entire system is fully observable and determined and we have achieved comparable numerical results as those reported by works with other types of measurement, e.g.,\cite{singh2009measurement,wang2004revised}.  
\end{remark}

It is also worthwhile to mention that, DSSE focusing on voltage magnitudes has significant practical meanings, especially considering that most works on optimization and control in distribution systems aim at voltage magnitudes regulation; see, e.g., a recent survey \cite{molzahn2017survey}. Fast and reliable voltage magnitudes estimation results can bridge the gap between the optimization and control designs that require a global view of the entire distribution systems and the reality of distribution systems equipped with limited reliable measurement devices; see, e.g., our recent work \cite{guo2019solving}.

The rest of this paper is organized as follows. Section~\ref{sec:model} models the distribution systems, formulates DSSE as a WLS problem, and provides observability analysis. Section~\ref{sec:hierarchical} proposes a multi-area algorithm to solve the WLS problem based on the gradient algorithm, which is extended to DSSE in multi-phase systems in Section~\ref{sec:multiphase}. Section~\ref{sec:num} provides numerical results and Section~\ref{sec:conclude} concludes this paper. 

\subsection*{Notations}
In this paper, we use bold upper-case letters to represent matrices, e.g., $\mathbf{A}$, italic bold letters to represent vectors, e.g, $\bm{A}$ and $\bm{a}$, and non-bold letters to represent {scalars}, e.g., $A$ and $a$. Superscript $^{\top}$ performs vector or matrix transpose. $|\cdot|$ denotes the cardinality of a set. $[\cdot]_{\Omega}$ makes projection upon set $\Omega$. Operator $\bigtimes$ represents the Cartesian product of sets. $\mathfrak{i}:=\sqrt{-1}$ is used as the imaginary unit.


\section{Distribution System State Estimation}\label{sec:model}

\subsection{WLS Estimator}

Denote by $\bm{z}$ the true system states, $\bm{y}$ the measurement, $\bm{h}(\bm{z})$ the nonlinear relationship between $\bm{z}$ and $\bm{y}$ with measurement error $\bm{\xi}$ following normal distribution $\cN(\bm{0},\mathbf{\Sigma})$ with zero mean and covariance matrix $\mathbf{\Sigma}$:
\begin{eqnarray}
\bm{y}=\bm{h}(\bm{z})+\bm{\xi}.\label{eq:hz}
\end{eqnarray}
Let $\mathbf{W}=\mathbf{\Sigma}^{-1}$ be a weight matrix and the WLS estimator can be formulated as \cite{wu1990power, dehghanpour2019survey, gomez2004power}:
\begin{eqnarray}
&\underset{\bm{z}\in\Omega}{\min}&\frac{1}{2} \big(\bm{y}-\bm{h}(\bm{z})\big)^{\top}\mathbf{W}\big(\bm{y}-\bm{h}(\bm{z})\big),\label{eq:se0}
\end{eqnarray}
where $\Omega$ denotes a \textit{convex} and \textit{compact} set of reasonable range for estimated states.

\subsection{Single-Phase Power Flow Model}
Consider a radial power distribution network denoted by a directed tree graph $\cT=\{\cN \cup \{0\},\cE\}$, with $N+1$ nodes collected in the set $\cN \cup \{0\}$, where $\cN:=\{1, ..., N\}$ and node $0$ is the slack bus, and distribution lines collected in the set $\cE$. For each node $i\in \hN$, denote by $\cE_i \subseteq \cE$ the set of lines on the unique path from node $0$ to node $i$, and let $p_i$ and $q_i$ denote the real and reactive power injected to node $i$, where negative power injection means power consumption and positive power injection means power generation. Let $v_i$ be the squared magnitude of the complex voltage (phasor) at node $i$. For each line $(i, j)\in \cE$, denote by $r_{ij}$ and $x_{ij}$ its resistance and reactance, and  $P_{ij}$ and $Q_{ij}$ the sending-end real and reactive power from node $i$ to node $j$. Let $\ell_{ij}$ denote the squared magnitude of the complex branch current (phasor) from node $i$ to $j$. We adopt the following DistFlow model \cite{baran1989optimala, baran1989optimalb} for the radial distribution network:
\begin{subequations}\label{eq:bfm}
	\begin{eqnarray}
	\hspace{-5mm} P_{ij} \hspace{-2mm} &=&\hspace{-2mm} - p_j +\hspace{-2mm}\sum_{k: (j,k)\in \cE} \hspace{-2mm} P_{jk}+  r_{ij}  \ell_{ij}  \label{p_balance}, \\
	\hspace{-5mm}Q_{ij} \hspace{-2mm}&=&\hspace{-2mm}  -q_j + \hspace{-2mm}\sum_{k: (j,k)\in \cE} \hspace{-2mm} Q_{jk} + x_{ij} \ell_{ij} \label{q_balance},\\
	\hspace{-5mm}v_j \hspace{-2mm}&=&\hspace{-2mm}  v_i - 2 \big(r_{ij} P_{ij} + x_{ij} Q_{ij} \big) + \big(r_{ij}^2+x_{ij}^2\big) \ell_{ij} \label{v_drop},\\[1pt]
	\hspace{-5mm}\ell_{ij}v_i \hspace{-2mm}&=&\hspace{-2mm}   P_{ij}^2 + Q_{ij}^2  \label{currents}.
	\end{eqnarray}
\end{subequations}

\subsubsection{Linearized Power Flow}
Following \cite{baran1989network,farivar2013equilibrium}, we assume that the active and reactive power loss $r_{ij} \ell_{ij}$ and $ x_{ij} \ell_{ij}$, as well as $r^2_{ij} \ell_{ij}$ and $x^2_{ij} \ell_{ij}$, are negligible and can thus be ignored. Indeed, the losses are much smaller than power flows $P_{ij}$ and $Q_{ij}$, typically on the order of $1\%$ \cite{farivar2013equilibrium,DistributionLosses}. Moreover, such modeling discrepancy will be largely mitigated by the feedback information from the nonlinear power flow model to be introduced in Section~\ref{sec:gradient}. Eqs.~\eqref{eq:bfm} thus become the following linear equations:
\begin{subequations}\label{eq:bfm2}
	\begin{eqnarray}
	\hspace{-5mm} P_{ij} \hspace{-2mm} &=&\hspace{-2mm} - p_j +\hspace{-2mm}\sum_{k: (j,k)\in \cE} \hspace{-2mm} P_{jk}  \label{p_balance2}, \\
	\hspace{-5mm}Q_{ij} \hspace{-2mm}&=&\hspace{-2mm}  -q_j + \hspace{-2mm}\sum_{k: (j,k)\in \cE} \hspace{-2mm} Q_{jk}  \label{q_balance2},\\
	\hspace{-5mm}v_j \hspace{-2mm}&=&\hspace{-2mm}  v_i - 2 \big(r_{ij} P_{ij} + x_{ij} Q_{ij} \big)  \label{v_drop2}.
	\end{eqnarray}
\end{subequations}

By \cite{farivar2013equilibrium}, Eqs.~\eqref{eq:bfm2} generate the following linear model for the squared voltage magnitudes and the nodal power injections:
\begin{eqnarray}
\bm{v}&=&\mathbf{R}\bm{p}+\mathbf{X}\bm{q}+\tilde{\bm{v}},\label{eq:lindistflow}
\end{eqnarray}
where bold symbols $\bm{v}=[v_1,\ldots,v_N]^{\top}$, $\bm{p}=[p_1,\ldots,p_N]^{\top}$, $\bm{q}=[q_1,\ldots,q_N]^{\top}\in\mathbb{R}^N$ represent vectors; $\tilde{\bm{v}}$ is a constant vector depending on initial conditions; and the sensitivity matrices $\mathbf{R}$ and $\mathbf{X}$, respectively, consists of elements:
\begin{eqnarray}
R_{ij}:= \!\!\! \sum_{(h,k)\in \cE_i \cap \cE_j}\!\!\!\! 2\cdot r_{hk}, \ \ \ \ X_{ij}:=\!\!\!\! \sum_{(h,k)\in \cE_i \cap \cE_j}\!\!\!\! 2\cdot x_{hk}  \label{X_def}.
\end{eqnarray}

Here, the voltage-to-power-injection sensitivity factors $R_{ij}$ ($X_{ij}$) are obtained through \textit{the resistance (resp. reactance) of the common path of node $\bm{i}$ and $\bm{j}$ leading back to node 0}. Keep in mind that this result serves as the basis for designing the hierarchical distributed algorithm to be introduced later. Fig.~\ref{fig:RX} illustrates $\cE_i \cap \cE_j$ for two arbitrary nodes $i$ and $j$ in a radial network and their corresponding $R_{ij}$ and $X_{ij}$.

\begin{figure}[h]
	\centering
	\includegraphics[width=.33\textwidth]{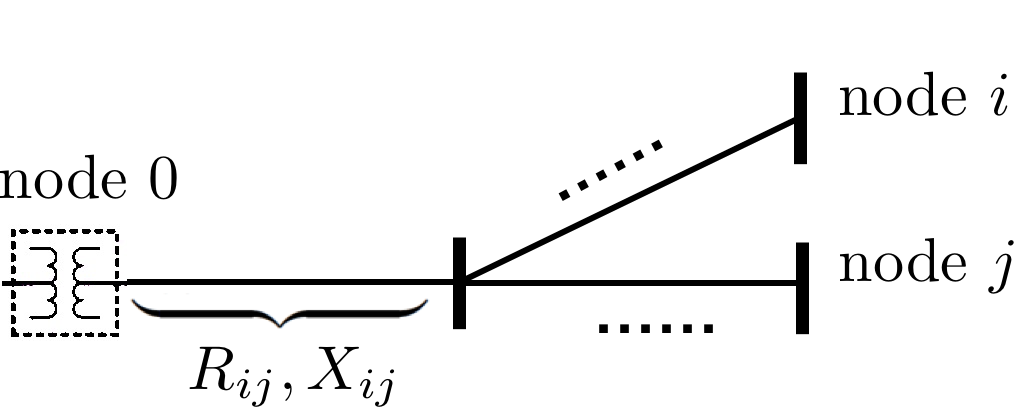}
	\caption{$\cE_i \cap \cE_j$ for  two arbitrary nodes $i, j$ in the network and the corresponding mutual voltage-to-power-injection sensitivity factors $R_{ij}, X_{ij}$. }
	\label{fig:RX}
\end{figure}

\subsection{Distribution System State Estimation Problem}\label{sec:formulation}
In the modeled distribution network, we choose $\bm{z}=[\bm{p}^{\top},\bm{q}^{\top}]^{\top}\in\mathbb{R}^{2N}$ as the system states. Notice that the entire system is uniquely determined by \eqref{eq:bfm} given $\bm{p}$ and $\bm{q}$.

We consider $\cN$ to be the union of two nonoverlapping subsets: the subset of the zero-injection nodes $\cN_0$ and the subset of the remaining load nodes $\cN_L$. Denote by $\cM_p, \cM_q\subset\cN_L$ the subsets of load nodes with (pseudo) measurable quantities of $[\hat{p}_i]^{\top}_{i\in\cM_p}$ and $[\hat{q}_i]^{\top}_{i\in\cM_q}$, and $\cM_v\subset\cN$ the subset of nodes with real-time voltage measurement $[\hat{v}_i]^{\top}_{i\in\cM_v}$. Here, the voltage magnitudes $[\hat{v}_i]^{\top}_{i\in\cM_v}$ are usually measured  by voltage magnitude meters and (phasor measurement units) PMUs with relatively high accuracy, and the power injections $[\hat{p}_i]^{\top}_{i\in\cM_p}$ and $[\hat{q}_i]^{\top}_{i\in\cM_q}$ are usually based on pseudo-measurements from {historical} data or learning techniques with potentially large errors\footnote{While traditional distribution system usually have a limited number of reliable measuring devices, smart measuring devices at the load side are now increasingly available to provide accurate real-time load measurement. Such additional measurement is consistent with the proposed problem formulation and solution method. For example, we can use smaller $\sigma^2_{p_i}$ and $\sigma^2_{q_i}$ for real-time load measurement than those for pseudo-measurement in \eqref{eq:WLS}.}. Let $m=|\cM_p|+|\cM_q|+|\cM_v|$  with $|\cdot|$ denoting the cardinality of a set. 

We further assume that measurement errors are independent. As a result, $\mathbf{\Sigma}$ is a diagonal matrix written as:
\begin{eqnarray}
\mathbf{\Sigma}\hspace{-2mm}&=&\hspace{-2mm}\diag\Big\{\Big[[\sigma^2_{p_i}]^{\top}_{i\in\cM_p},[\sigma^2_{q_i}]^{\top}_{i\in\cM_q},[\sigma^2_{v_i}]^{\top}_{i\in\cM_v}\Big]^{\top}\Big\}\in\mathbb{R}_+^{m\times m},\nonumber
\end{eqnarray}
where $\sigma_{p_i}$, $\sigma_{q_i}$, and $\sigma_{v_i}$ are the standard deviations of measurement errors for $p_i$, $q_i$, and $v_i$, respectively. Then, we recast problem \eqref{eq:se0} based on the linearized power flow as follows:
\begin{subequations}\label{eq:se}
	\begin{eqnarray}
	&\hspace{-7mm}\underset{\bm{p},\bm{q},\bm{v}}{\min}&\hspace{-2mm}\sum_{i\in\cM_p}\hspace{-2mm}\frac{(p_i-\hat{p}_i)^2}{2\sigma_{p_i}^2}+\hspace{-2mm}\sum_{i\in\cM_q}\hspace{-2mm}\frac{(q_i-\hat{q}_i)^2}{2\sigma_{q_i}^2}+\hspace{-2mm}\sum_{i\in\cM_v}\hspace{-2mm}\frac{(v_i-\hat{v}_i)^2}{2\sigma_{v_i}^2},\label{eq:WLS}\\
	&\hspace{-7mm}\text{s.t.}& \hspace{0mm}\bm{v}=\mathbf{R}\bm{p}+\mathbf{X}\bm{q}+\tilde{\bm{v}},\label{eq:voltage}\\
	&\hspace{-7mm}& (\bm{p},\bm{q})\in\Omega.\label{eq:Omega}
	\end{eqnarray}
\end{subequations}
Here, $\Omega={\bigtimes}_{i\in\cN}\Omega_i$ is a convex and compact set with:
\begin{eqnarray}
\Omega_i\hspace{-2mm}&:=&\hspace{-2mm}\big\{(p_i,q_i)\ |\ p^{\text{min}}_i\leq p_i\leq p^{\text{max}}_i,\ q^{\text{min}}_i\leq q_i\leq q^{\text{max}}_i\big\},\nonumber
\end{eqnarray}
providing a reasonable estimation range for node $i$. To better facilitate estimation, such feasible sets should be loose enough to cover load/generation capacity while tight enough to avoid unreasonable/overfitting results. For example, for a load node, the upper bound can be set to twice the normal peak load, and the lower bound can be set to zero. For the zero-injection node $i\in\cN\backslash\cN_L$, we set its $\Omega_i:=\{(0,0)\}$ as a singleton.

\subsection{Observability Analysis}

Following \cite{giannakis2013monitoring,wu1985network}, we define the \textit{observability} of a distribution network as the ability to \emph{uniquely} identify the state $(\bm{p},\bm{q})$. Observability is usually analyzed based on a linearized model \cite{wu1985network}. Therefore, we consider the following linearized equation for \eqref{eq:hz}:
\begin{eqnarray}
\bm{y}=\mathbf{H}\bm{z}+\bm{\xi}.\label{eq:Hz}
\end{eqnarray}
\begin{definition}[100\% Observability]
	The network is said to be fully (100\%) observable if for all $\bm{z}$ such that $\mathbf{H}\bm{z}=\bm{0}$, there must be $\bm{z}=\bm{0}$. Otherwise, any $\bm{z}'$ that satisfies $\bm{z}'\neq \bm{0}$ and $\mathbf{H}\bm{z}'=\bm{0}$ is called an unobservable state. 
\end{definition}

\begin{theorem}\label{the:observe}
	A sufficient condition for the distribution network $\cT$ to be fully (100\%) observable is $\cM_p=\cM_q=\cN_L$.
\end{theorem}
For a detailed analysis, refer to Appendix~\ref{sec:app_observe}.

\subsubsection{Pseudo-measurement and zero-injection nodes}
By Theorem~\ref{the:observe}, to suffice 100\% observability and for notational simplicity, we expand $\cM_p$ and $\cM_q$ to $\cN_L$ for the rest of this paper. Notice that by setting $\cM_p=\cM_q=\cN_L$, we do not assume new information because pseudo-measurements based on historical data with large $\sigma_{p_i}$ and $\sigma_{q_i}$ can be applied to all $i\in\cN_L$.

\subsubsection{Discussion} Observability is usually an issue in distribution systems due to a limited number of sensors deployed; however, by Theorem~\ref{the:observe}, as long as we have pseudo-measurements for all nodes, we achieve 100\% observability. This is further eased because not all nodes in a distribution system have power injection. Take the test case in Section~\ref{sec:num} as an example: of all 4,512 nodes on the primary side, only 1,335 nodes---less than 30\% of all nodes---have loads, leaving the rest of the nodes' power injections accurately known to be zero. This observation not only makes 100\% observability easier to achieve but also foreshadows the surprisingly high accuracy of DSSE in realistic settings even with very few accurate measurement nodes, as will be shown by the numerical examples in Section~\ref{sec:num}.


\section{Hierarchical Distributed State Estimation}\label{sec:hierarchical}

Traditionally, the Gauss-Newton method is applied to solve \eqref{eq:se} \cite{gomez2004power,gomez2011multilevel}. In large distribution networks, however, Gauss-Newton method may experience the following issues. 1) The calculation of the (approximated) Gauss-Newton step can be computationally heavy. 2) Gauss-Newton method usually requires the problem to be unconstrained, or the initialization is close to the optimal. Here, we have an additional constraint \eqref{eq:Omega} and we do not have any assumption on decent initial points. 3) The accuracy of the Gauss-Newton method may be questionable when there are limited reliable measurements, as will be shown by the numerical results in Section~\ref{sec:num}. To address these issues, we apply the projected gradient algorithm to solve \eqref{eq:se} instead. The gradient-based method will also lead to an equivalent multi-area distributed implementation that simplifies computation and accelerates convergence for large distribution systems while providing robust and accurate estimation results.

\subsection{Gradient Algorithm}\label{sec:gradient}
Problem~\eqref{eq:se} may feature large difference among weights $\sigma^2_{p_i}$, $\sigma^2_{q_i}$, and $\sigma^2_{v_i}$. For example, if we use pseudo-measurement of 50\% standard deviation from its real values for some $p_i$, and 1\% standard deviation from the true value for some $v_i$, we end up with $1/2500$ ratio among their weights. Such ill-conditioned WLS may lead to inaccurate results generated by Gauss-Newton method, especially when matrix inverse is performed. Therefore, we next consider applying a more robust gradient algorithm to solve problem~\eqref{eq:se} instead.

Substitute the equality constraint \eqref{eq:voltage} into the objective function \eqref{eq:WLS}, denote by 
$\nu_j(t)=(v_j(t)-\hat{v}_j)/{\sigma^2_{v_j}},\forall j\in\cM_v$
for notational simplicity, and we have the following iterative projected gradient algorithm to solve \eqref{eq:se}:
\begin{subequations}\label{eq:grad}
\begin{eqnarray}
\hspace{-6mm}p_i(s+1)\hspace{-3mm}&=&\hspace{-3mm}\Big[p_i(s)-\epsilon\Big(\hspace{-2mm}\sum_{j\in\cM_v}\hspace{-1.5mm} R_{ji} \nu_j(s) +\frac{p_i(s)-\hat{p}_i}{\sigma^2_{p_i}}\Big)\Big]_{\Omega_i}\!\!\!,\nonumber\\
&&\hspace{45mm}i\in\cN_L,\label{eq:gradp}\\
\hspace{-6mm}q_i(s+1)\hspace{-3mm}&=&\hspace{-3mm}\Big[q_i(s)-\epsilon\Big(\hspace{-2mm}\sum_{j\in\cM_v}\hspace{-1.5mm}X_{ji}\nu_j(s)+\frac{q_i(s)-\hat{q}_i}{\sigma^2_{q_i}}\Big)\Big]_{\Omega_i}\!\!\!,\nonumber\\
&&\hspace{45mm}i\in\cN_L,\label{eq:gradq}\\
\hspace{-6mm}&&\hspace{-20mm}\{p_i(s+1),q_i(s+1)\}\in\Omega_i, \hspace{23mm} i\in\cN\backslash\cN_L,\\
\hspace{-6mm}\bm{v}(s+1)\hspace{-3mm}&=&\hspace{-3mm}\mathbf{R}\bm{p}(s+1)+\mathbf{X}\bm{q}(s+1)+\tilde{\bm{v}},\label{eq:gradv}
\end{eqnarray}
\end{subequations}
where $s$ is the iteration index, and $[\ ]_{\Omega_i}$ denotes the projection operator upon the feasible set $\Omega_i$.
Recall that, for zero-injection node $i\in\cN\backslash\cN_L$, we have $\Omega_i=\{(0,0)\}$. 

Because the cost function \eqref{eq:WLS} is strongly convex, asymptotic linear convergence to the optimal can be readily shown and omitted here. Interested readers are referred to \cite{boyd2004convex}. 

\begin{remark}[Nonlinear Power Flow]\label{mark:nonlinear}
	Note that the convex WLS problem~\eqref{eq:se} and the pertinent gradient algorithm~\eqref{eq:grad} estimate system states based on the lossless linearized power flow model \eqref{eq:bfm2}, which inevitably causes an additional discrepancy between the estimates and the true states based on the nonlinear power flow model \eqref{eq:bfm}. Such modeling errors can be mitigated with model-based feedback implementation---i.e., we can replace \eqref{eq:gradv} with nonlinear power flow equations---whereas the gradient steps \eqref{eq:gradp}--\eqref{eq:gradq} are still based on the fixed $\mathbf{R}$ and $\mathbf{X}$ matrices. In this way, we are able to achieve fast implementation by avoiding computing the state-based linearization model each iteration, while sacrificing only a small amount of accuracy due to modeling error. We refer to Corollary~\ref{the:feedback} in Section~\ref{sec:app_converge} (as a special case of Theorem~\ref{the:converge}) of Appendix~\ref{sec:app_converge} for the analytical characterization of the feedback-based algorithm. Numerical results in Section~\ref{sec:num} validate that this model-based feedback implementation achieves accurate estimation.
\end{remark}

\subsection{Subtree-Based Multi-Area Structure}
When implementing \eqref{eq:grad}, a DSO with global network structure $\mathbf{R}, \mathbf{X}$ and voltage measurements needs to calculate $\sum_{j\in\cM_v}R_{ji} \nu_j(s)$ and $\sum_{j\in\cM_v}X_{ji} \nu_j(s)$ for all $i\in\cN$ at each iteration $s$. Assume that the number of nodes with voltage measurements are proportional to the size of $\cN$. Then, the computational complexity of \eqref{eq:grad} at each iteration is proportional to $N^2$, rapidly increasing as the size of the distribution system scales up and prohibiting fast state estimation.

To improve computation efficiency and facilitate real-time DSSE, we next propose a hierarchical distributed implementation of \eqref{eq:grad} that can reduce the repetitive computation and distribute the computation burden with a subtree-based multi-area network structure. In doing so, we accelerate the DSSE without compromising any accuracy compared with centralized implementation. For this purpose, we first define subtrees as follows.

\begin{definition}[Subtree]
	A subtree of a tree $\cT$ is a tree consisting of a node in $\cT$, all its descendants in $\cT$, and their connecting lines.
\end{definition}

Next, we divide the distribution network $\cT$ into $K$ nonoverlapping subtrees indexed by $\cT_k=\{\cN_k,\cE_k\}$ with root node $n_k^0$, $k\in\cK=\{1,\ldots,K\}$, based on certain pre-defined criteria (e.g., geographic or administrative structure), and the remaining area. Here, $\cN_k$ of size $N_k$ is the set of nodes in subtree $\cT_k$, and $\cE_k$ contains their connecting lines. We denote the set of the remaining ``unclustered" node by $\cN_0$. Thus, we have  $\cup_{k\in\cK}\cN_k\cup\cN_0=\cN\cup\{0\}$ and $\cN_j \cap \cN_k=\emptyset,\forall j\neq k$.
We assign AMS $k$ that is cognizant of the topology of $\cT_k$ to collect pseudo-measurements within $\cT_k$ and communicate with the DSO to collaboratively execute \eqref{eq:grad}, i.e., to estimate system states for the entire network.  The next Lemma plays a key role in the design that follows in Section~\ref{sec:hier}.

\begin{lemma}[Lemma~3 in \cite{zhou2019accelerated}]\label{lem:commonpath}
	Given any two subtrees $\cT_h$ and $\cT_k$ with their root nodes $n^0_h$ and $n^0_k$, we have $R_{ij}=R_{n^0_h n^0_k},X_{ij}=X_{n^0_h n^0_k}$, for any $i\in\cN_h$ and any $j\in\cN_k$. Similarly, given any unclustered node $i\in\cN_0$ and a subtree $\cT_k$ with its root node $n^0_k$, we have $R_{ij}=R_{i n^0_k},X_{ij}=X_{i n^0_k}$, for any $j\in\cN_k$.
\end{lemma}

\begin{figure}[h]
	\centering
	\includegraphics[scale=0.3]{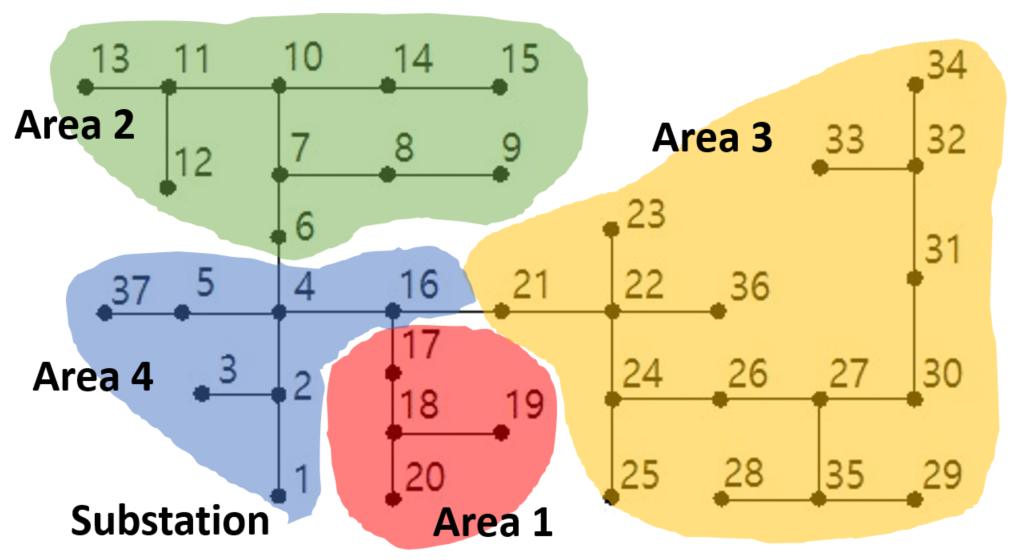}
	\caption{Illustration of the subtree-featuring multi-area structure based on IEEE 37-node test feeder. Areas 1--3 have subtree topology and Area~4 is the remaining area.}\label{fig:testfeeder37}
\end{figure}

Lemma~\ref{lem:commonpath} is illustrated in Fig.~\ref{fig:testfeeder37}, where any node $i$ in subtree Area~2 and any node $j$ in subtree Area~3 share the identical common path leading back to the substation, i.e., lines 1-2-4. Therefore, $R_{ij}=r_{12}+r_{24}$, and $X_{ij}=x_{12}+x_{24}$ for any $i$ within Area~2 and any $j$ within Area~3. Also see Fig.~\ref{fig:testfeeder} for a large system illustration with Areas~1--4 being nonoverlapping subtrees. As an example, any node in Area~1 and any node in Area~3 (both marked by blue triangles) share the same common path (marked by the yellow dashed line) leading back to the substation, i.e., they share the same voltage-to-power-injection sensitivity factor.

\subsection{Hierarchical Distributed Algorithm}\label{sec:hier}

Based on Lemma~\ref{lem:commonpath}, we can decompose the coupling terms (ignoring iteration index) in \eqref{eq:gradp} for $i\in\cN$ as follows:
\begin{eqnarray}
\sum_{j\in\mathcal{M}_v}R_{ji}\nu_j:=\alpha_i\hspace{50mm}\nonumber\\
=\begin{cases}
\underbrace{\underset{j\in\mathcal{M}_v\cap\mathcal{N}_k}{\sum}\hspace{-4mm}R_{ji}\nu_j}_{:=\alpha_{k,i}^{\text{in}}}+\underbrace{\hspace{-3mm}\underset{h\in\mathcal{K},h\neq k}{\sum}\hspace{-3mm}R_{n_h^0 n_k^0}\hspace{-3mm}\underset{j\in\mathcal{M}_v\cap\cN_h}{\sum}\hspace{-3mm}\nu_j+\hspace{-3mm}\underset{j\in\cM_v\cap\cN_0}{\sum}\hspace{-3mm}R_{jn_k^0}\nu_j}_{:=\alpha_k^{\text{out}}},\\
\hspace{12mm}\hspace{35mm}\text{if }i\in\cN_k,\ k\in\cK;\\
\hspace{0mm}\underset{h\in\mathcal{K}}{\sum}R_{n_h^0 i} \hspace{-2mm}\underset{j\in\mathcal{M}_v\cap\cN_h}{\sum}\hspace{-3mm}\nu_j+\hspace{-3mm}\underset{j\in\cM_v\cap\cN_0}{\sum}\hspace{-3mm}R_{ji}\nu_j,\hspace{13mm}\text{if }i\in\cN_0.
\end{cases}\hspace{-5mm}\label{eq:decompose}
\end{eqnarray}
For completeness, we present the results of decomposing $\sum_{j\in\mathcal{M}_v} \hspace{-1mm}X_{ji}\nu_j(t)$ for \eqref{eq:gradq} as:
\begin{eqnarray}
\sum_{j\in\mathcal{M}_v}X_{ji}\nu_j:=\beta_i\hspace{50mm}\nonumber\\
=\begin{cases}
\underbrace{\underset{j\in\mathcal{M}_v\cap\mathcal{N}_k}{\sum}\hspace{-4mm}X_{ji}\nu_j}_{:=\beta_{k,i}^{\text{in}}}+\underbrace{\hspace{-3mm}\underset{h\in\mathcal{K},h\neq k}{\sum}\hspace{-3mm}X_{n_h^0 n_k^0}\hspace{-3mm}\underset{j\in\mathcal{M}_v\cap\cN_h}{\sum}\hspace{-3mm}\nu_j+\hspace{-3mm}\underset{j\in\cM_v\cap\cN_0}{\sum}\hspace{-3mm}X_{jn_k^0}\nu_j}_{:=\beta_k^{\text{out}}},\\
\hspace{12mm}\hspace{35mm}\text{if }i\in\cN_k,\ k\in\cK;\\
\hspace{0mm}\underset{h\in\mathcal{K}}{\sum}X_{n_k^0 i}\hspace{-2mm}\underset{j\in\mathcal{M}_v\cap\cN_h}{\sum}\hspace{-3mm}\nu_j+\hspace{-3mm}\underset{j\in\cM_v\cap\cN_0}{\sum}\hspace{-3mm}X_{ji}\nu_j,\hspace{13mm}\text{if }i\in\cN_0.
\end{cases}\hspace{-5mm}\label{eq:decompose2}
\end{eqnarray}
Based on \eqref{eq:decompose}, we can design a hierarchical distributed implementation of \eqref{eq:grad}, put as Algorithm~\ref{alg:distalg}. 

\begin{algorithm}[h]
	\caption{Gradient-Based Multi-Area DSSE} 
	\begin{algorithmic}\label{alg:distalg}	
		\REPEAT		
		\STATE 1) 
		AMS $k\in\cK$ calculates and sends $\hspace{-4mm}\underset{i\in\cM_v\cap\cN_k}{\sum}\hspace{-4mm}\nu_i(s)$ to DSO; DSO updates $\nu_i(s)$ for unclustered node $i\in\cN_0$.
		
		\STATE 2) DSO computes $\alpha_k^{\text{out}}(s)$ and $	\beta_k^{\text{out}}(s)$ for $k\in\cK$, and $\alpha_i(s)$ and $\beta_i(s)$ for $i\in\cN_0$ by \eqref{eq:decompose}--\eqref{eq:decompose2}, and sends $(\alpha_k^{\text{out}}(s),\beta_k^{\text{out}}(s))$ to AMS $k\in\cK$.
			
		\STATE 3) AMS $k\in\cK$ calculates $\alpha^{\text{in}}_{k,i}(s)$, $\beta^{\text{in}}_{k,i}(s)$, $\alpha_{i}(s)$, and $\beta_{i}(s)$ for $i\in\cN_k$ by \eqref{eq:decompose}--\eqref{eq:decompose2}.
		
		\STATE 4) AMS and DSO update $(p_i(s+1),q_i(s+1))$ for nodes managed by them by
		\begin{subequations}
		\begin{eqnarray}
		\hspace{-2mm}p_i(s+1)\hspace{-3mm}&=&\hspace{-3mm}\big[p_i(s)-\epsilon\big(\alpha_i(s)+(p_i(s)-\hat{p}_i)/{\sigma^2_{p_i}}\big)\big]_{\Omega_i},\nonumber\\
		\hspace{-2mm}q_i(s+1)\hspace{-3mm}&=&\hspace{-3mm}\big[q_i(s)-\epsilon\big(\beta_i(s)+(q_i(s)-\hat{q}_i)/{\sigma^2_{q_i}}\big)\big]_{\Omega_i},\nonumber
		\end{eqnarray}
		\end{subequations}
		for $i\in\cN_L$ and $\{p_i(s+1),q_i(s+1)\}\in\Omega_i$ for $i\in\cN\backslash\cN_L$.
		
		\STATE 5) AMS send results from 4) to DSO and DSO uses grid simulator to update $\bm{v}(s+1)$ by \eqref{eq:gradv}.
		
		\UNTIL Stopping criterion is met (e.g., $\|\bm{v}(s+1)-\bm{v}(s)\|<\delta$ for some small $\delta>0$).
	\end{algorithmic}
\end{algorithm}

Algorithm~\ref{alg:distalg} presents an equivalent hierarchical distributed implementation of the gradient algorithm~\eqref{eq:grad} among the DSO and several AMS. At iteration $s$, {AMS} $k$ receives two {scalars}, $\alpha_k^{\text{out}}(s)$ and  $\beta_k^{\text{out}}(s)$, to update states vectors $[p_i]_{i\in\cN_k}$ and $[q_i]_{i\in\cN_k}$ within its area and sends back the results for the DSO to update the power flow. Moreover, such implementation also reduces the computational complexity when calculating the coupling terms; see the related discussion in Section~III-B of \cite{zhou2019accelerated}.

\begin{remark}[Asynchronous Updates] In practice, however, Algorithm~\ref{alg:distalg}, like most distributed algorithms, may experience a communication loss, local agents malfunction, etc., that result in asynchronous updates. Generally speaking, if none of the agents permanently loses communication with the rest of the system, Algorithm~\ref{alg:distalg} with asynchronous updates will eventually converge to the same results as the synchronous updates, with potentially longer convergence time. Detailed discussions are beyond the scope of this paper. We refer interested readers to \cite{bertsekas1989parallel, zhou2019online} for the detailed performance characterization of  distributed algorithms with asynchronous updates.
\end{remark}

\subsection{Real-Time State Estimation}
In reality, the system states and measured values could change rapidly from one second to the next. This leads to the following time-varying WLS optimization problem at time $t$:
\begin{subequations}\label{eq:set}
	\begin{eqnarray}
	&\hspace{-9mm}\underset{\bm{p}^t,\bm{q}^t,\bm{v}^t}{\min}&\hspace{-4mm}\sum_{i\in\cN_L}\hspace{-2mm}\frac{(p_i^t-\hat{p}_i^t)^2}{2\sigma_{p_i}^2}+\hspace{-2mm}\sum_{i\in\cN_L}\hspace{-2mm}\frac{(q_i^t-\hat{q}_i^t)^2}{2\sigma_{q_i}^2}+\hspace{-2mm}\sum_{i\in\cM_v}\hspace{-2mm}\frac{(v_i^t-\hat{v}_i^t)^2}{2\sigma_{v_i}^2},\label{eq:WLSt}\\
	&\hspace{-9mm}\text{s.t.}& \hspace{-2mm}\bm{v}^t=\mathbf{R}\bm{p}^t+\mathbf{X}\bm{q}^t+\tilde{\bm{v}},\label{eq:voltaget}\\
	&\hspace{-9mm}&\hspace{-2mm} (\bm{p}^t,\bm{q}^t)\in\Omega^t,\label{eq:Omegat}
	\end{eqnarray}
\end{subequations}
with pseudo-measurement $\hat{p}_i^t$, $\hat{q}_i^t$, and $\hat{v}_i^t$, and $\Omega^t$ potentially updated at time $t$.

To solve the time-varying WLS problem, we can implement Algorithm~\ref{alg:distalg} at each $t$. In a large network with fast updating measurements, however, solving \eqref{eq:set} for each $t$ is stressful for  computational resources. A computationally tractable alternative is to execute a limited number of gradient steps from Algorithm~\ref{alg:distalg} at each $t$. Particularly, if the duration from $t$ to $t+1$ is so short that only one gradient step may be implemented as follows:
\begin{subequations}\label{eq:gradt}
	\begin{eqnarray}
	\hspace{-5mm}p_i^{t+1}\hspace{-3mm}&=&\hspace{-3mm}\Big[p_i^t-\epsilon\Big(\hspace{-2mm}\sum_{j\in\cM_v}\hspace{-1.5mm} R_{ji} \nu_j^t +\frac{p_i^t-\hat{p}_i^t}{\sigma^2_{p_i}}\Big)\Big]_{\Omega^{t+1}_i},\ i\in\cN_L\label{eq:gradpt}\\
	\hspace{-5mm}q_i^{t+1}\hspace{-3mm}&=&\hspace{-3mm}\Big[q_i^t-\epsilon\Big(\hspace{-2mm}\sum_{j\in\cM_v}\hspace{-1.5mm}X_{ji}\nu_j^t+\frac{q_i^t-\hat{q}^t_i}{\sigma^2_{q_i}}\Big)\Big]_{\Omega^{t+1}_i},\ i\in\cN_L\label{eq:gradqt}\\
	\hspace{-6mm}&&\hspace{-14mm}\{p_i^{t+1},q_i^{t+1}\}\in\Omega_i^{t+1}, \hspace{30mm} i\in\cN\backslash\cN_L,\\
	\hspace{-5mm}\bm{v}^{t+1}\hspace{-3mm}&=&\hspace{-3mm}\mathbf{R}\bm{p}^{t+1}+\mathbf{X}\bm{q}^{t+1}+\tilde{\bm{v}},\label{eq:gradvt}
	\end{eqnarray}
\end{subequations}
with $\nu_j^t:=(v_j^t-\hat{v}^t_j)/{\sigma^2_{v_j}},\ \forall j\in\cM_v$. Asymptotic convergence of \eqref{eq:gradt} toward a bounded ball around the optimal of \eqref{eq:set} at time $t$ can be shown, given that the change of true states from one time to the next is bounded. For better readability, see Theorem~\ref{the:converge} in Appendix~\ref{sec:app_converge} for the detailed convergence analysis of the proposed online algorithm.

The real-time implementation~\eqref{eq:gradt} enables the proposed algorithm to track the fast-changing system states in distribution systems with deep penetration of intermittent DERs without waiting for convergence at each scenario, as will be shown in Section~\ref{sec:num_realtime}.


\section{Multi-Phase Multi-Area State Estimation}\label{sec:multiphase}
This section extends the multi-area DSSE design to multi-phase unbalanced distribution systems.

\subsection{Multi-Phase System Modeling}
Define the imaginary unit $\mathfrak{i}:=\sqrt{-1}$. Let $a,b,c$---we use $a=0$, $b=1$, and $c=2$ when calculating phase difference---denote the three phases, and $\Phi_i$ the set of phase(s) of node $i\in\cN$, e.g., $\Phi_i=\{a,b,c\}$ for a three-phase node $i$, and $\Phi_j=\{b\}$ for a single b-phase node $j$. Also, in a three-phase system, one usually has $\Phi_0=\{a,b,c\}$ at the root node. Define $\cN^{\phi}\subseteq\cN$ as the subset of $\cN$ collecting nodes that have phase $\phi$. Denote by $p_i^{\phi}$, $q_i^{\phi}$, $V_i^{\phi}$ and $v_i^{\phi}$ the real power injection, the reactive power injection, the complex voltage phasor, and the squared voltage magnitude, respectively, of node $i\in\cN$ at phase $\phi\in\Phi_i$. 
Denote by $N_{\Xi}:=\sum_{i\in\cN}|\Phi_i|=\sum_{\phi\in\Phi_0}|\cN^{\phi}|$ the total cardinality of the multi-phase system, where $|\cdot|$ calculates the cardinality of a set.

Let $z^{\varphi\phi}_{\zeta\xi}\in\mathbb{C}$ be the complex (mutual) impedance of line $(\zeta,\xi)\in\cE$ between phase $\phi$ and $\varphi$. For example, for a three-phase line $(\zeta,\xi)\in\cE$:
	\begin{eqnarray}
	z_{\zeta\xi}=\begin{bmatrix}z^{aa}_{\zeta\xi}& z^{ab}_{\zeta\xi}&z^{ac}_{\zeta\xi}\\ z^{ba}_{\zeta\xi}& z^{bb}_{\zeta\xi}&z^{bc}_{\zeta\xi}\\z^{ca}_{\zeta\xi}& z^{cb}_{\zeta\xi}&z^{cc}_{\zeta\xi}\end{bmatrix}\in\mathbb{C}^{3\times 3}.\nonumber
	\end{eqnarray}
We construct 
$Z^{\varphi\phi}_{ij}=\underset{(\zeta,\xi)\in\cE_i \cap \cE_j}{\hspace{-4mm}\sum}\hspace{-4mm}z^{\varphi\phi}_{\zeta\xi}\hspace{2mm}\in\mathbb{C}$
as the aggregate impedance (if $\varphi=\phi$) or mutual impedance (if $\varphi\neq\phi$) of the common path of node $i$ and $j$ leading back to node 0, and $\overline{Z}^{\varphi\phi}_{ij}$ its conjugate. 

We denote by $\bm{v}_{_{\Xi}}=[[v_1^{\phi}]^{\top}_{\phi\in\Phi_1}, \ldots, [v_{N}^{\phi}]^{\top}_{\phi\in\Phi_N}]^{\top}\in\mathbb{R}^{N_{\Xi}}$ the multi-phase squared voltage magnitude vector, and $\bm{p}_{_{\Xi}}=[[p_1^{\phi}]_{\phi\in\Phi_1}^{\top}, \ldots, [p_{N}^{\phi}]^{\top}_{\phi\in\Phi_N}]^{\top}\in\mathbb{R}^{N_{\Xi}}$ and $\bm{q}_{_{\Xi}}=[[q_1^{\phi}]^{\top}_{\phi\in\Phi_1}, \ldots, [q_{N}^{\phi}]^{\top}_{\phi\in\Phi_N}]^{\top}\in\mathbb{R}^{N_{\Xi}}$ the multi-phase power injection vectors. We then extend the linearization \eqref{eq:lindistflow} to its multi-phase counterpart, written as:
\begin{eqnarray}
\bm{v}_{_{\Xi}}=\mathbf{R}_{_{\Xi}}\bm{p}_{_{\Xi}}+\mathbf{X}_{_{\Xi}}\bm{q}_{_{\Xi}}+\tilde{\bm{v}}_{_{\Xi}},\label{eq:lindistflow_phi}
\end{eqnarray}
where $\tilde{\bm{v}}_{_{\Xi}}\in\mathbb{R}^{N_{\Xi}}$ is a constant vector depending on the squared voltage magnitudes at all phases of the slack bus, and the voltage-to-power sensitivity matrices $\mathbf{R}_{_{\Xi}}, \mathbf{X}_{_{\Xi}}\in\mathbb{R}^{N_{\Xi}\times N_{\Xi}}$ are determined by the linear approximation method developed for multi-phase system \cite{gan2014convex,gan2016online} comprising elements calculated as follows:
	\begin{eqnarray}\label{eq:RX3}
	\partial_{p_j^{\phi}}v_i^{\varphi} \hspace{-2mm} =  2\mathfrak{Re} \big\{ \overline{Z}^{\varphi\phi}_{ij}  \omega^{\varphi-\phi}\big\},\ \ 
	\partial_{q_j^{\phi}}v_i^{\varphi} \hspace{-2mm}  = \hspace{0mm} -2 \mathfrak{Im} \big\{ \overline{Z}^{\varphi\phi}_{ij}  \omega^{\varphi-\phi}\big\},
	\end{eqnarray}
for any $\varphi\in\Phi_i$, $\phi\in\Phi_j$, $i,j\in\cN$, with  $\omega=e^{-\mathfrak{i} 2\pi/3}$, and $\mathfrak{Re}\{\cdot\}$ and $\mathfrak{Im}\{\cdot\}$ denoting the real and imaginary parts of a complex number. Note that when $\varphi=\phi$, Eqs.~\eqref{eq:RX3} coincide with $R_{ij}$ and $X_{ij}$ in Eqs.~\eqref{X_def} for any nodes $i,j\in\cN$; otherwise, when $\varphi\neq\phi$, Eqs.~\eqref{eq:RX3} calculate the aggregate mutual impedance---rotated by the phase difference $\pm 2\pi/3$---of the common path of nodes $i,j\in\cN$ leading back to node 0.

\begin{remark}[Unbalanced Nonlinear Power Flow]
The linearization \eqref{eq:lindistflow_phi}--\eqref{eq:RX3} are based on the assumptions that 1) the system is lossless and 2) the three phases are nearly balanced, i.e., $2\pi/3$ apart \cite{gan2014convex,gan2016online}, introducing modeling error. Nevertheless, echoing our previous discussion in Remark~\ref{mark:nonlinear}, we will  implement the designed algorithm with feedback voltage updated from unbalanced nonlinear power flow to reduce the modeling error. 
\end{remark}

\subsection{Solving Multi-Phase State Estimation Problem}
We use $\bm{v}_{_{\Xi}}(\bm{p}_{_{\Xi}},\bm{q}_{_{\Xi}})$ to represent Eqs.~\eqref{eq:lindistflow_phi}, and formulate the DSSE problem for the multi-phase system as follows:
\begin{subequations}\label{eq:multiSE}
\begin{eqnarray}
&\underset{\bm{p}_{\Xi},\bm{q}_{\Xi},\bm{v}_{\Xi}}{\min}&\sum_{i\in\cN_L}\sum_{\phi\in\Phi_i}\left(\frac{(p^{\phi}_i-\hat{p}^{\phi}_i)^2}{2\sigma_{p^{\phi}_i}^2}+\frac{(q^{\phi}_i-\hat{q}^{\phi}_i)^2}{2\sigma_{q_i^{\phi}}^2}\right)\nonumber\\[-2pt]
&&\hspace{3mm}+\sum_{i\in\cM_v}\sum_{\phi\in\Phi_i}\frac{(v^{\phi}_i-\hat{v}^{\phi}_i)^2}{2\sigma_{v_i^{\phi}}^2},\\
&\text{s.t.}& \bm{v}_{_{\Xi}}=\mathbf{R}_{_{\Xi}}\bm{p}_{_{\Xi}}+\mathbf{X}_{_{\Xi}}\bm{q}_{_{\Xi}}+\tilde{\bm{v}}_{_{\Xi}},\label{eq:ss}\\
&& (p_i^{\phi},q_i^{\phi})\in\Omega_i^{\phi},\ \phi\in\Phi_i,\ i\in\cN,
\end{eqnarray}
\end{subequations}
where $\Omega_i^{\phi}$ is a convex and compact set presenting the upper and lower bounds for reasonable estimation of  $(p_i^{\phi},q_i^{\phi})$. 

Note that the multi-phase sensitivity matrices $R_{_{\Xi}}$ and $X_{_{\Xi}}$ in Eqs.~\eqref{eq:RX3} have similar structures as their single-phase counterparts $\mathbf{R}$ and $\mathbf{X}$ defined in Eqs.~\eqref{X_def}, i.e., the values of $	\partial_{p_j^{\varphi}}v_i^{\phi}$ and $\partial_{q_j^{\varphi}}v_i^{\phi}$ for any $i,j\in\cN$ depend only on the common path of $i$ and $j$ leading back to node 0, adjusted by their angle difference $\phi-\varphi$. This motivates us to design a similar multi-area implementation for the multi-phase DSSE. Denote by $\nu_j^{\varphi}(t)=(v^{\varphi}_j(t)-\hat{v}^{\varphi}_j)/{\sigma^2_{v^{\varphi}_j}}$ for notational simplicity, and the gradient algorithm for solving \eqref{eq:multiSE} is:
\begin{subequations}\label{eq:gradient3}
\begin{eqnarray}
\hspace{-7mm}p^{\phi}_i(s+1)\hspace{-1mm}&=&\hspace{-1mm}\Big[p^{\phi}_i(s)-\epsilon\Big(\sum_{j\in\cM_v}\sum_{\varphi\in\Phi_j}\partial_{p_i^{\phi}}v_j^{\varphi}\nu_j^{\varphi}(s)\nonumber\\
\hspace{-7mm}&&\hspace{-9mm}+\big(p^{\phi}_i(s)-\hat{p}^{\phi}_i\big)/\sigma^2_{p^{\phi}_i}\Big)\Big]_{\Omega^{\phi}_i},\ \phi\in\Phi_i, i\in\cN_L,\hspace{-10mm}\label{eq:multip}\\
\hspace{-7mm}q^{\phi}_i(s+1)\hspace{-1mm}&=&\hspace{-1mm}\Big[q^{\phi}_i(s)-\epsilon\Big(\sum_{j\in\cM_v}\sum_{\varphi\in\Phi_j}\partial_{q_i^{\phi}}v_j^{\varphi}\nu_j^{\varphi}(s)\nonumber\\
\hspace{-3mm}&&\hspace{-9mm}+\big(q^{\phi}_i(s)-\hat{q}^{\phi}_i\big)/\sigma^2_{q^{\phi}_i}\Big)\Big]_{\Omega^{\phi}_i},\ \phi\in\Phi_i, i\in\cN_L,\hspace{-10mm}\label{eq:multiq}\\
&&\hspace{-23mm}\{p^{\phi}_i(s+1),q^{\phi}_i(s+1)\}\in\Omega^{\phi}_i, \hspace{7mm}\ \phi\in\Phi_i, i\in\cN\backslash\cN_L,\label{eq:multipq}\\
\hspace{-7mm}\bm{v}_{_{\Xi}}(s+1)\hspace{-1mm}&=&\hspace{-1mm}\mathbf{R}_{_{\Xi}}\bm{p}_{_{\Xi}}(s+1)+\mathbf{X}_{_{\Xi}}\bm{q}_{_{\Xi}}(s+1)+\tilde{\bm{v}}.\label{eq:multiflow}
\end{eqnarray}
\end{subequations}

Then we decompose the coupling term (ignoring iteration index) in Eq.~\eqref{eq:multip} as:
\begin{eqnarray}
\hspace{-13mm}&&\sum_{j\in\cM_v}\sum_{\varphi\in\Phi_j}\partial_{p_i^{\phi}}v_j^{\varphi}\nu_j^{\varphi}\nonumber\\
\hspace{-13mm}&=&\hspace{-2mm}
\begin{cases}
2\mathfrak{Re} \Big\{ \underset{\varphi\in\Phi_0}{\sum}\omega^{\varphi-\phi}\Big(\hspace{-1mm}\underset{\substack{j\in\cN^{\varphi}\cap\\\cN_k\cap\cM_v}}{\sum}\hspace{-2mm}  \overline{Z}^{\varphi\phi}_{ji} \nu_j^{\varphi}+\hspace{-2mm}\underset{\substack{n_h^0 \in \cN^\varphi\\h\in\cK,h\neq k }}{\sum}\hspace{-2mm}\overline{Z}^{\varphi\phi}_{n_h^0 n_k^0}\hspace{-2mm}\underset{\substack{j\in\cN^{\varphi}\cap\\\cN_h\cap\cM_v}}{\sum}\hspace{-2mm}\nu_j^{\varphi}\\
\hspace{10mm}+\underset{\substack{j\in\cN^{\varphi}\cap\\\cN_0\cap\cM_v}}{\sum}\hspace{-2mm}\overline{Z}^{\varphi\phi}_{j n_k^0}\nu_j^{\varphi}\Big)\Big\},\hspace{10mm} \text{if}~i\in\cN_k,\ k\in\cK;\vspace{2mm}\\[-4pt]
2\mathfrak{Re} \Big\{ \underset{\varphi\in\Phi_0}{\sum}\omega^{\varphi-\phi}\Big(\hspace{-1mm}\underset{\substack{n_h^0 \in \cN^\varphi\\h\in\cK }}{\sum}\hspace{-3mm}\overline{Z}^{\varphi\phi}_{n_h^0 i} \hspace{-2mm}\underset{\substack{j\in\cN^{\varphi}\cap\\\cN_h\cap\cM_v}}{\sum}\hspace{-2mm}\nu_j^{\varphi}+\hspace{-2mm}\underset{\substack{j\in\cN^{\varphi}\cap\\\cN_0\cap\cM_v}}{\sum}\hspace{-2mm}\overline{Z}^{\varphi\phi}_{j i}\nu_j^{\varphi}\Big)\Big\},\\ \hspace{60mm}\text{if}~i\in\cN_0.
\end{cases}\hspace{-10mm}\label{eq:decompose3}
\end{eqnarray}
We can apply similar approaches to decompose the coupling term in Eq.~\eqref{eq:multiq} for reactive power updates. Like Algorithm~\ref{alg:distalg}, we can implement \eqref{eq:gradient3} in a hierarchical distributed way based on Eqs.~\eqref{eq:decompose3}. The resultant design is also ready for real-time implementation, such as \eqref{eq:gradt}. The convergence analysis for the single-phase system in Appendix~\ref{sec:app_converge} completely applies here.

\begin{figure}[h]
	\centering
	\includegraphics[scale=0.34]{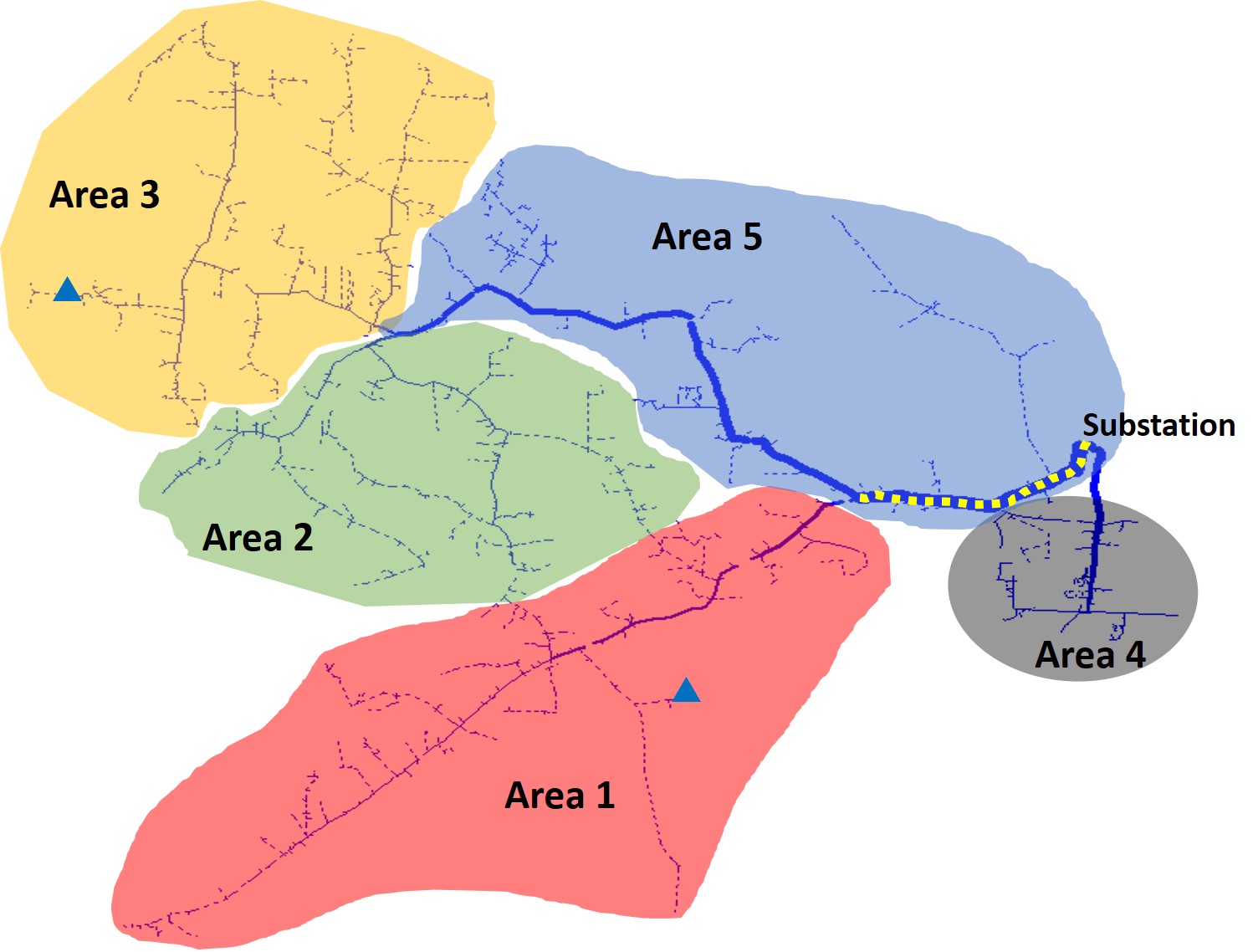}
	\caption{The 11,000-node test feeder constructed from the IEEE 8,500-node test feeder and EPRI Test Circuit Ckt7. Four subtree-based areas, indexed by Area 1--4, together with the remaining area, indexed by Area 5, are formed for our experiments. The yellow dashed line is the common path for Area 1 and Area 3.}\label{fig:testfeeder}
\end{figure}


\section{Numerical Results}\label{sec:num}
\subsection{IEEE 37-Node System Tests}\label{sec:compare}
We first use the {single-phase} IEEE 37-node test feeder (Fig.~\ref{fig:testfeeder37}) to compare the performance between the proposed gradient method with the traditional Gauss-Newton method. Three nodes---6, 12, 34---are selected to measure the voltage magnitudes with 1\% standard deviation from the actual value, and all 36 load nodes have pseudo-measurements with 50\% standard deviation from their true values\footnote{We only use pseudo-measurement for all loads and a small percentage of nodes with voltage magnitudes measurement in simulations to provide conservative performance evaluation, which can be further improved by deploying more advanced measuring techniques and more measuring devices in practice.}. The nonlinear power flow model and the Gauss-Newton state estimation are based on MATPOWER 7.0 \cite{zimmerman2010matpower}. We do not use the multi-area implementation for the gradient method because the network is small. We run this setup 10,000 times under both methods, and we record the results in TABLE~\ref{table:newtonresult}. We also plot the histogram of the average error distribution of the 10,000 cases for both methods in Fig.~\ref{fig:hist} (upper). The error rates are calculated by comparing the estimated results against the ground truth values.

The results show that the gradient algorithm has a slower but comparable computational time in this small network, but it consistently generates much more accurate and robust estimation results with limited reliable measurements. 

{Next, we extend the results to a large test system to demonstrate the scalability, consistency, and accuracy of the proposed algorithm, compared with Gauss-Newton method.}

\begin{table}[t]
	\begin{center}
		\begin{tabular}{|| c | c  | c ||}   	
			\hline
			& Gradient Method & G-N Method\\
			\hline
			Average Time (s) & 0.0249 & 0.0112  \\
			\hline		
			Average Error & 0.4\%  & 0.91\%  \\
			\hline
			Average Max. Error & 0.83\% & 2.19\%  \\
			\hline
		\end{tabular}
		\caption{Performance comparison between gradient method and Gauss-Newton method based on 10,000 tests with random realization of voltage measurement and pseudo-measurement noises on IEEE 37-node test feeder.}\label{table:newtonresult}
	\end{center}
\end{table}

\begin{figure}[h]
	\centering
	\includegraphics[scale=0.35]{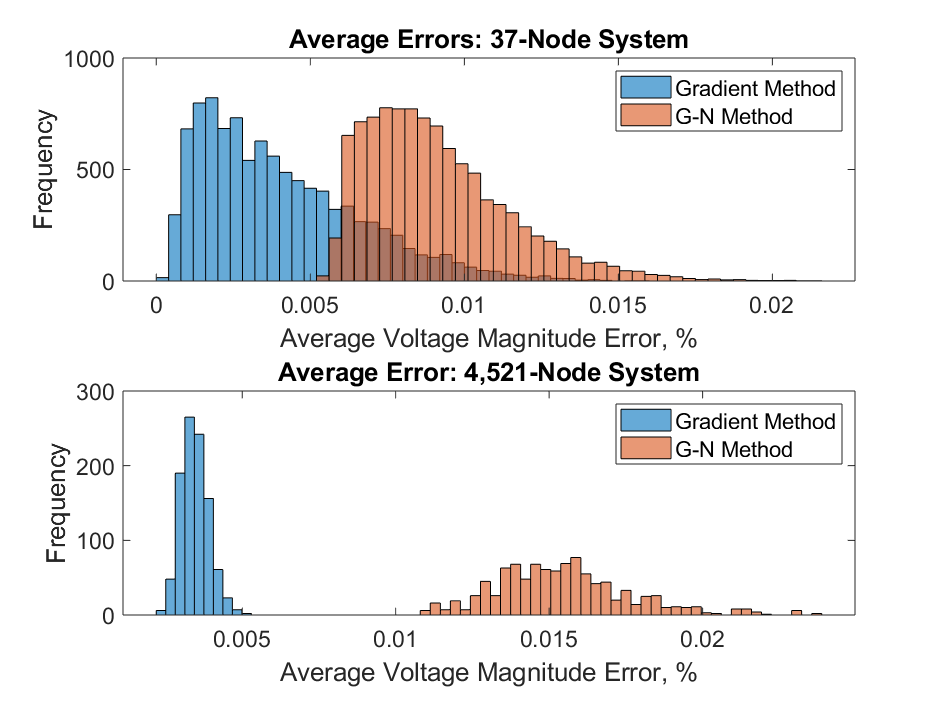}
	\caption{Histograms of the average voltage magnitude estimation errors of the gradient method and the Gauss-Newton method in 37-node system with 10,000 test cases (upper) and in 4,521-node system with 1,000 test cases (lower).}\label{fig:hist}
\end{figure}

\subsection{Large System Tests}\label{sec:num_setup}
\subsubsection{System Setup}
A three-phase, unbalanced, 11,000-node test feeder is constructed by connecting an IEEE 8,500-node test feeder and a modified EPRI Ckt7 at the substation.  Fig.~\ref{fig:testfeeder} shows the single-line diagram of the feeder, where the line width is proportional to the nominal power flow on it. The primary side of the feeder is modeled in detail, whereas the loads on the secondary side (which in this system is the aggregation of several loads) are lumped into corresponding distribution transformers, resulting in a 4,521-node network with 1,043 aggregated load nodes to be estimated. We group all the nodes into four areas marked in Fig.~\ref{fig:testfeeder}: Area 1 contains 357 load nodes, Area 2 contains 222, Area 3 contains 310, and Area 4 contains 154.

We implement the multi-phase multi-area DSSE algorithm based on \eqref{eq:gradient3} with \eqref{eq:multip}--\eqref{eq:multipq} carried out in a hierarchical way and \eqref{eq:multiflow}
replaced with the three-phase unbalanced nonlinear power flow simulated in OpenDSS. The simulation is conducted on a laptop with Intel Core i7-7600U CPU @ 2.80GHz 2.90GHz, 8.00GB RAM, running Python 3.6 on Windows 10 Enterprise Version.

We randomly select 5\% of the nodes within Area 1--4 (i.e., 3.6\% of all 4,521 nodes) to measure their voltage magnitudes. with measurement errors subject to normal distribution of zero mean and 1\% standard deviation from their true values. All 1,043 loads nodes have pseudo-measurement whose errors are subject to normal distribution of zero mean and 50\% standard deviation from their true values.

\begin{figure}
	\centering
	\includegraphics[scale=0.45]{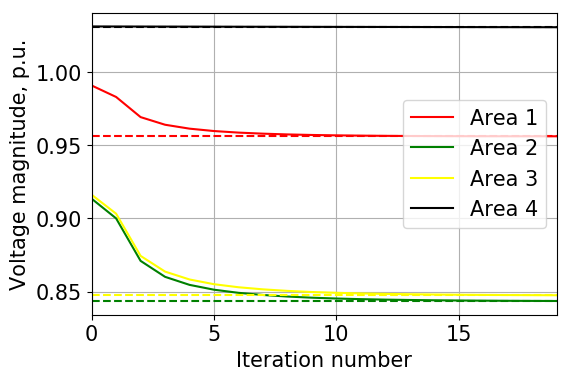}
	\caption{The gradient-based DSSE converges within 20 iterations.}\label{fig:converge}
\end{figure}

\begin{figure}
	\centering
	\includegraphics[scale=0.35]{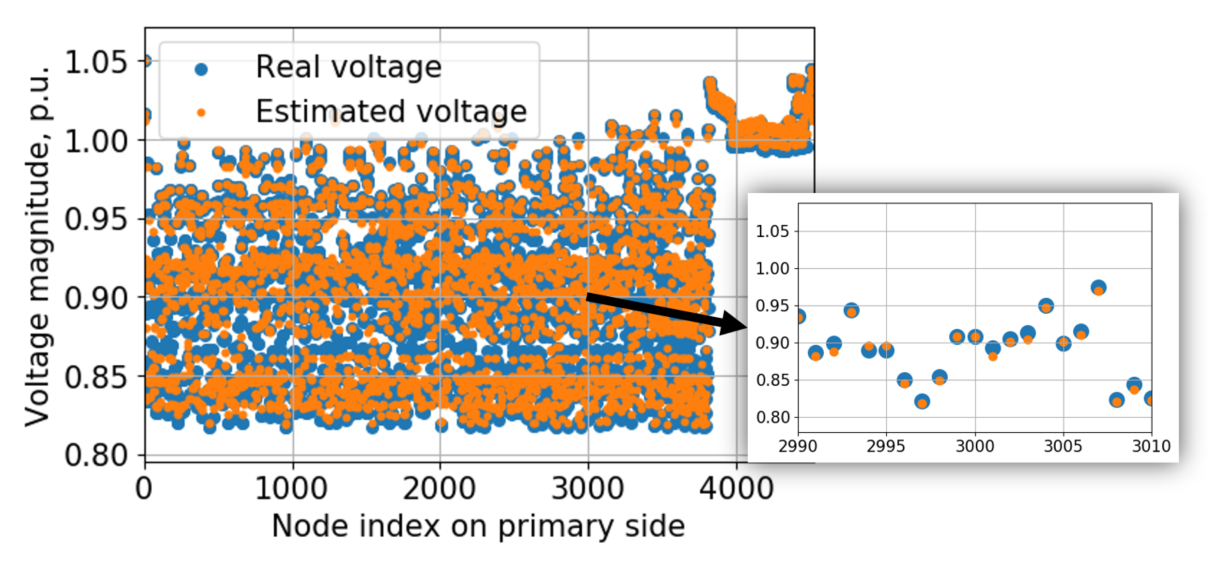}
	\caption{The estimated voltage magnitudes are close to the true values after 20 iterations. We zoom into the range of nodes 2990--3010 to show detailed estimation result.}\label{fig:dot}
\end{figure}

\begin{table}[h]
	\begin{center}
		
		\begin{tabular}{|| c | c  | c ||}   	
			\hline
			& Gradient Method & G-N Method\\
			\hline
			Average Time (s)  & 54.05 & 164.84  \\
			\hline		
			Average Mag. Error & 0.35\%  & 1.57\%  \\
			\hline
			Average Max. Mag. Error & 1.34\% & 6.29\%  \\
			\hline
			Average Ang. Error & 1.75$^\degree$  & 4.15$^\degree$  \\
			\hline
			Average Max. Ang. Error & 5.24$^\degree$ & 16.51$^\degree$  \\
			\hline
		\end{tabular}
		\caption{Performance comparison between the proposed gradient method and Gauss-Newton method based on 1,000 tests with random realization of voltage measurement location and measurement noises on the 4,521-node test feeder.}\label{table:newtonresult2}
	\end{center}
\end{table}

\subsubsection{Convergence and Accuracy}
As shown in Fig.~\ref{fig:converge}, it takes less than 20 iterations for the gradient-based algorithm to converge to close to the true states, which are marked by dashed lines. Fig.~\ref{fig:dot} shows the estimated voltage magnitudes against the true voltage magnitudes after 20 iterations.

We run the simulation 1,000 times with random realization of voltage measurement location selection and pseudo-measurement errors. The average results of the 1,000 cases are presented as follows: 1) the average voltage magnitude estimation error per node is 0.35\%, 2) the maximal nodal voltage estimation error is 1.34\%, and 3) the average time to compute 20 iterations is 54.05 seconds if calculated centrally and 21.31 seconds (approximately 1 second per iteration) if parallel computation among the four areas are implemented. Also note that by design the centrally calculated estimation results are \emph{identical} to the results of the distributed algorithm, but the former takes approximately 2.5 times longer to compute in this network. {We put results in TABLE~\ref{table:newtonresult2}, together with voltage angle estimation results for completeness.}


\subsubsection{Comparison against Gauss-Newton Method}
{We run the Gauss-Newton method based on the same system model and measurement for 1,000 times and record the results as follows: 1) average voltage magnitude estimation error per node is 1.54\%, 2) the maximal nodal voltage estimation error is 6.29\%, and 3) the average time to converge is 367.27 seconds. We summarize the comparison in TABLE~\ref{table:newtonresult2}, and plot the histogram of the average errors comparison in Fig.~\ref{fig:hist} (lower). As we can see, the gradient-based state estimation method illustrates much better scalability, and much higher accuracy than Gauss-Newton method under such scenario of limited measurement resources. The longer computation time for Newton-Gauss method is expected because Newton method is known to be unscalable in large systems; the higher error rates may be caused by the inaccurate matrix inversion operation conducted in Gauss-Newton method for the ill-conditioned WLS formulation in distribution systems with smaller errors for voltage measurement and much larger errors for pseudo-measurement. Angle estimation results also show big difference between the two algorithms. We will explore the reasons and related deeper meanings in our ongoing and future works.}

\begin{figure}[t]
	\centering
	\includegraphics[scale=0.3]{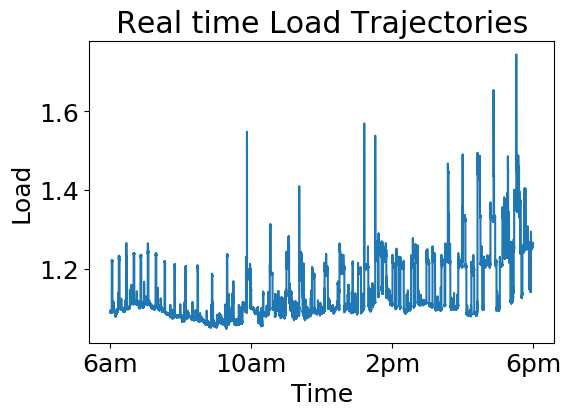}
	\includegraphics[scale=0.3]{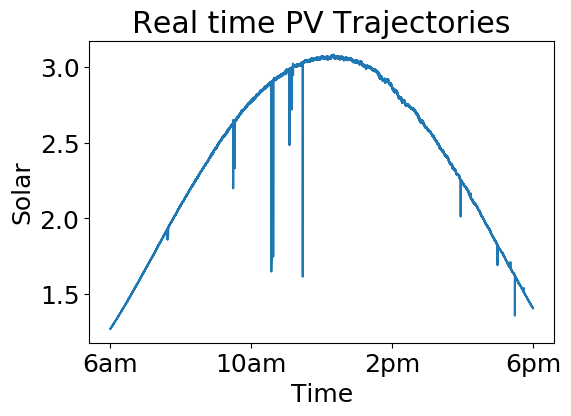}
	\caption{Real-time load (left) and PV generation (right) profile from 6 a.m. to 6 p.m. with 1 second temporal granularity.}\label{fig:load}
\end{figure}

\begin{figure*}
	\centering
	\includegraphics[scale=0.34]{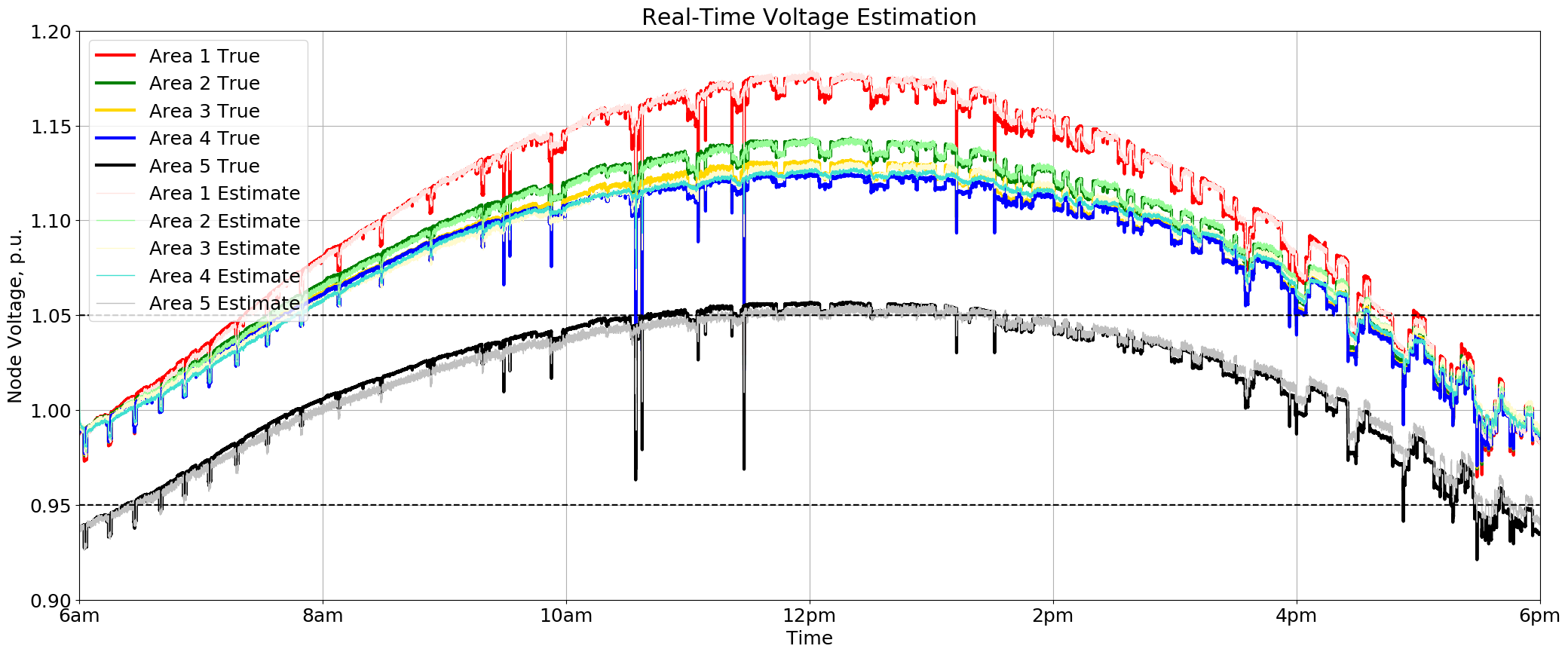}
	\caption{Real-time DSSE from 6 a.m. to 6 p.m. with one gradient step update every second.}\label{fig:realtimese}
\end{figure*}

\begin{figure*}
	\centering
	\includegraphics[scale=0.34]{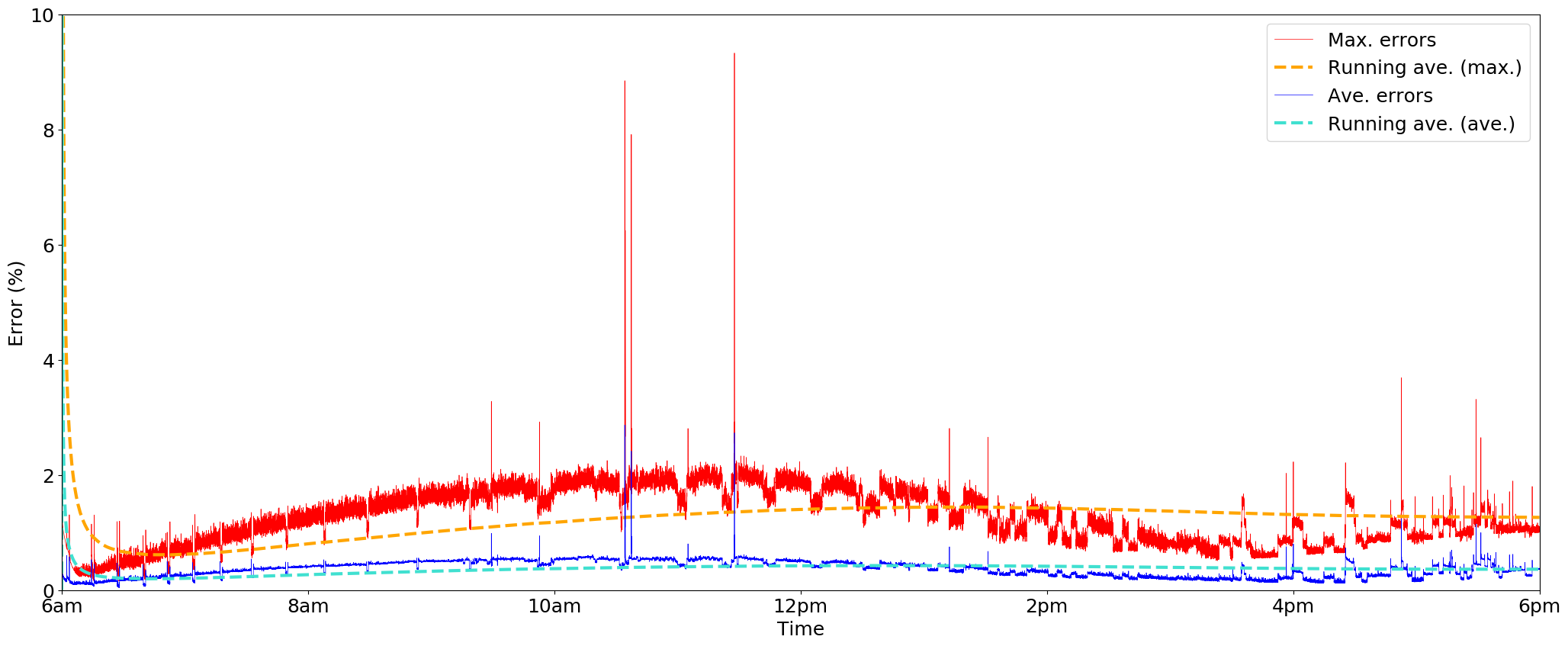}
	\caption{Average errors and maximal errors at every second and their respective running average from 6 a.m. to 6 p.m. }\label{fig:realtimeerr}
\end{figure*}

\subsection{Real-Time State Estimation}\label{sec:num_realtime}
{The one-iteration-per-second computational speed of the proposed multi-area algorithm encourages us to implement the algorithm in a real-time manner.} We use real load and solar irradiance data measured from feeders in Anatolia, California, during a week of
August 2012 \cite{Bank13} with one-second temporal granularity; see Fig.~\ref{fig:load} for the time-varying load and PV generation profile from 6 a.m. to 6 p.m., totaling 43,200 consecutive scenarios used for real-time DSSE. With the same system setup described in Section~\ref{sec:num_setup}, {i.e., real-time pseudo-measurement for all loads with measurement error of 50\% standard deviation, and voltage measurement at 3.6\% of all nodes with measurement error of 1\%,} we implement one gradient step per second to track the fast time-varying voltage magnitudes, and we plot the voltage magnitudes of five arbitrarily chosen nodes from five areas in Fig.~\ref{fig:realtimese}, where the lines of darker colors represent the true voltages, and the lines of lighter colors are the estimated voltages. The real-time estimate is very close to the true values shown in Fig.~\ref{fig:realtimese}. Indeed, the running average of the node average error of all 43,200 cases approaches 0.36\%, and that of the maximal error approaches to 1.26\% after 12 hours, which is consistent with the offline results; see Fig.~\ref{fig:realtimeerr}. 

{Such results demonstrate that the proposed algorithm is capable of fast real-time DSSE with high accuracy in large distribution systems, given the accessibility of high-resolution data of future distribution networks. If the data is available at slower paces, then the proposed algorithm is allowed to execute more iterations until updated data arrive and thus generate better estimation accuracy. On the other hand, the Gauss-Newton method needs more than 10 seconds to calculate one iteration, unable to perform real-state estimation of finer time granularity.}


\section{Conclusions and Future Works}\label{sec:conclude}
We propose a multi-area state estimation algorithm for large distribution networks with rapidly changing system states. The algorithm enables distributed implementation of the gradient algorithm for solving a large WLS problem in a distributed way among AMS and DSO without compromising estimation accuracy. Featured with fast convergence and high accuracy, the algorithm can be implemented in a real-time manner to monitor fast-changing system states. Numerical results based on a large multi-phase distribution system and real-time data are provided to validate the performance of the design. Comparison against traditional Gauss-Newton method shows the major advantages of the proposed algorithm in terms of computation time and estimation accuracy.

In future work, we will continue to explore the gradient-based multi-area DSSE algorithms to incorporate more types of measurements and states to be estimated, e.g., voltage phases, currents, and power branch flows, to further improve the estimation results. 


\section*{Acknowledgments}
We thank Dr. Ahmed Zamzam at National Renewable Energy Laboratory for meaningful discussions and the anonymous reviewers for constructive suggestions to help improve this work.
	
This work was authored in part by the National Renewable Energy Laboratory, operated by Alliance for Sustainable Energy, LLC, for the U.S. Department of Energy (DOE) under Contract No. DE-EE-0007998. Funding provided by U.S. Department of Energy Office of Energy Efficiency and Renewable Energy Solar Energy Technologies Office. The views expressed in the article do not necessarily represent the views of the DOE or the U.S. Government. The U.S. Government retains and the publisher, by accepting the article for publication, acknowledges that the U.S. Government retains a nonexclusive, paid-up, irrevocable, worldwide license to publish or reproduce the published form of this work, or allow others to do so, for U.S. Government purposes.

\bibliographystyle{IEEEtran}
\bibliography{biblio.bib}


\appendices

\section{Observability Analysis}\label{sec:app_observe}
To {analyze} observability, we next calculate $\mathbf{H}$ based on \eqref{eq:bfm2}.
Let $\mathbf{B}\in\mathbb{R}^{(N+1) \times N}$ denote the incidence matrix of the network defined as:
\begin{eqnarray}
\begin{cases}
	B_{ie}=1,& \text{if } e=i\rightarrow j \text{ is a line,}\\
	B_{ie}=-1,& \text{if } e=k\rightarrow i \text{ is a line,}\\
	B_{ie}=0,& \text{otherwise,}
\end{cases}\nonumber
\end{eqnarray}
and let $\tilde{\mathbf{B}}\in\mathbb{R}^{N\times N}$ denote the reduced incidence matrix by removing the first row $\bm{B}_0^{\top}$ from $\mathbf{B}$, i.e.,
$\mathbf{B}=\left[\begin{matrix}\bm{B}_0^{\top}\\ \tilde{\mathbf{B}}\\\end{matrix}\right]$.
It can be derived from linearized power flow \eqref{eq:bfm2} that:
\begin{eqnarray}
\left[\begin{matrix}p_0\\\bm{p}\\\end{matrix}\right]=\mathbf{B}\bm{P},\hspace{3mm} \left[\begin{matrix}q_0\\\bm{q}\\\end{matrix}\right]=\mathbf{B}\bm{Q},\hspace{3mm} \mathbf{B}^{\top}\left[\begin{matrix}v_0\\\bm{v}\\\end{matrix}\right]=\tilde{\mathbf{R}}\bm{P}+\tilde{\mathbf{X}}\bm{Q},\nonumber
\end{eqnarray}
where $\tilde{\mathbf{R}}$ and $\tilde{\mathbf{X}}$ are $N\times N$ diagonal matrices with diagonal terms being $r_{ij}$ and $x_{ij}$, respectively. Note that for a power distribution network with tree topology,  $\tilde{\mathbf{B}}$ is nonsingular. Then we can further obtain that: 
\begin{eqnarray}
\hspace{-8mm}&&p_0=-\bm{\mathit{1}}^{\top}\bm{p},\ \  \bm{P}=\tilde{\mathbf{B}}^{-1}\bm{p},\ \ 
q_0=-\bm{\mathit{1}}^{\top}\bm{q},\ \  \bm{Q}={\tilde{\mathbf{B}}}^{-1}\bm{q},\nonumber\\
\hspace{-8mm}&&\bm{v}={\tilde{\mathbf{B}}}^{-\top}\tilde{\mathbf{R}}{\tilde{\mathbf{B}}}^{-1}\bm{p}+{\tilde{\mathbf{B}}}^{-\top}\tilde{\mathbf{X}}{\tilde{\mathbf{B}}}^{-1}\bm{q}-{\tilde{\mathbf{B}}}^{-\top}\bm{B}_0v_0.\nonumber
\end{eqnarray}
Therefore, with $\bm{z}=[\bm{p}^{\top},\bm{q}^{\top}]^{\top}$ defined as the state vector, the measurement $\bm{y}$ as a subvector of $[\bm{v}^{\top},p_0,q_0,\bm{p}^{\top},\bm{q}^{\top},\bm{P}^{\top},\bm{Q}^{\top}]^{\top}$ can be written as:
$\bm{y}=\mathbf{H}\bm{z}$
where $\mathbf{H}$ is a constant block matrix comprising appropriate submatrices of  
$\tilde{\mathbf{B}}^{-1}$, ${\tilde{\mathbf{B}}}^{-\top}\tilde{\mathbf{R}}{\tilde{\mathbf{B}}}^{-1}$, ${\tilde{\mathbf{B}}}^{-\top}\tilde{\mathbf{X}}{\tilde{\mathbf{B}}}^{-1}$, $-\bm{\mathit{1}}^{\top}$,
as well as appropriate rows of the identity matrix $\mathbf{I}$, depending on the specific measurements available.

According to the definition of observability from \cite{wu1985network,gomez2004power}, the network is 100\% observable if the matrix $\mathbf{H}$ has full column rank. In practice, available measurements are often fewer than what is needed to achieve 100\% observability. In that case, an observability index can be calculated as the ratio between the dimension of the observable state space over the dimension of the entire state space. Note that the dimension of the unobservable state space equals the dimension of the null space of $\mathbf{H}$. Therefore, the dimension of the observable state space equals the rank of $\mathbf{H}$, and thus the network observability is:
\begin{equation}
\text{Observability}=\text{Rank}(\mathbf{H})/{2N}\times 100\%.\label{eq:observe}
\end{equation}
For example, given $\bm{y}=[[\hat{p}_i]^{\top}_{i\in\cM_p}, [\hat{q}_i]^{\top}_{i\in\cM_q}, [\hat{v}_i]^{\top}_{i\in\cM_v}]^{\top}$ we have:
\begin{eqnarray}
\mathbf{H}=\ \ 
\left[\begin{matrix}
{\tilde{\mathbf{I}}}_p&\mathbf{0}\\
\mathbf{0}&{\tilde{\mathbf{I}}}_q\\
\left({\tilde{\mathbf{B}}}^{-\top}\tilde{\mathbf{R}}{\tilde{\mathbf{B}}}^{-1}\right)_v & \left({\tilde{\mathbf{B}}}^{-\top}\tilde{\mathbf{X}}{\tilde{\mathbf{B}}}^{-1}\right)_v\end{matrix}\right]\label{eq:H}
\end{eqnarray}
where 
$({\tilde{\mathbf{B}}}^{-\top}\tilde{\mathbf{R}}{\tilde{\mathbf{B}}}^{-1})_v$ is the submatrix of ${\tilde{\mathbf{B}}}^{-\top}\tilde{\mathbf{R}}{\tilde{\mathbf{B}}}^{-1}$ 
comprising only the $|\cM_v|$ rows corresponding to the nodes that have measurements of squared voltage magnitudes; similarly for 
$\left({\tilde{\mathbf{B}}}^{-\top}\tilde{\mathbf{X}}{\tilde{\mathbf{B}}}^{-1}\right)_v$. ${\tilde{\mathbf{I}}}_p$
and ${\tilde{\mathbf{I}}}_q$ are submatrices of the identity matrix $\mathbf{I}$ comprising only the $|\cM_p|$ and $|\cM_q|$ rows corresponding to the nodes that have pseudo-measurements of active and reactive power injections, respectively. Note that matrix ${\tilde{\mathbf{B}}}^{-\top}\tilde{\mathbf{R}}{\tilde{\mathbf{B}}}^{-1}$ (resp. ${\tilde{\mathbf{B}}}^{-\top}\tilde{\mathbf{X}}{\tilde{\mathbf{B}}}^{-1})$ has the following structure: its $(i,j)$-th element equals $R_{ij}$ (resp. $X_{ij}$), which is the summation of $r_e$ (resp. $x_e$) over all the lines e on the path from the slack bus to the joint node of nodes $i$ and $j$ \cite{zhou2019accelerated}. The rank of $\mathbf{H}$ and the observability of the network can thus be directly calculated.

\section{Convergence Analysis}\label{sec:app_converge}
Denote by $\bm{z}^{t}=[\bm{p}^{t\top},\bm{q}^{t\top}]^{\top}$, by $\bm{z}^{t*}$ the optimal state estimation solution of problem \eqref{eq:set} at time $t$, and by $C^t(\bm{z}^t)$ the cost function in \eqref{eq:WLSt} after substituting \eqref{eq:voltaget} into the cost function. Rewrite the gradient algorithm \eqref{eq:gradt} as the following mapping for convenience:
\begin{eqnarray}
\bm{z}^{t+1} &=& \left[\bm{z}^t-\epsilon \bm{f}^t(\bm{z}^{t})\right]_{\Omega}\label{eq:mapping_l}
\end{eqnarray}
with $\bm{f}^t(\bm{z}^{t})=\nabla_{\bm{z}^t}C^t(\bm{z}^t)$. Denote by $\tilde{\bm{f}}^t(\bm{z}^{t})$ the counterpart of $\bm{f}^t(\bm{z}^{t})$ calculated based on voltage feedback from the nonlinear power flow, i.e., replacing \eqref{eq:gradvt} with \eqref{eq:bfm}. The real-time DSSE with nonlinear power flow feedback can be written in the form of the following dynamics:
\begin{eqnarray}
\bm{z}^{t+1} &=& \left[\bm{z}^t-\epsilon \tilde{\bm{f}}^t(\bm{z}^{t})\right]_{\Omega}.\label{eq:mapping_nl}
\end{eqnarray}

We proceed with the following reasonable assumptions for analytical characterization:
\begin{enumerate}
\item[\textit{A1.}] The difference between the optimal solutions of any two consecutive timeslots is bounded, i.e.:  
\begin{eqnarray}
\|\bm{z}^{t+1*}-\bm{z}^{t*}\| \leq \Delta_1,\ \forall t.\label{eq:delta1}
\end{eqnarray}
\item[\textit{A2.}] The discrepancy between the linearized power flow model and the original nonlinear power flow model is bounded for any feasible $\bm{z}$. As a result, we have:
\begin{eqnarray}
\|\tilde{\bm{f}}^t(\bm{z})-\bm{f}^t(\bm{z})\|_2^2\leq \Delta_2, \ \forall t.\label{eq:delta2}
\end{eqnarray}
\end{enumerate}
Additionally, the following results follow from the problem formulation (Lemma~1--2 in \cite{zhou2019accelerated}):
\begin{enumerate}
\item[\textit{B1.}] (Strongly monotone operator) There exists some constant $M>0$ such that for any feasible $\bm{z},\bm{z}'$ one has:
\begin{eqnarray}
(\bm{f}^t(\bm{z})-\bm{f}(\bm{z}'))^{\top}(\bm{z}-\bm{z}') \geq M\|\bm{z}-\bm{z}'\|_2^2.\label{eq:B1}
\end{eqnarray}
\item[\textit{B2.}] (Lipschitz continuity) There exists some constant $L>0$ such that for any feasible $\bm{z},\bm{z}'$ one has
\begin{eqnarray}
\|\bm{f}^t(\bm{z})-\bm{f}^t(\bm{z}')\|_2^2\leq L^2\|\bm{z}-\bm{z}'\|^2_2,\ \forall t.\label{eq:B2}
\end{eqnarray}
\item[\textit{B3.}] The following relation holds:
\begin{eqnarray}
M\leq L.\label{eq:B3}
\end{eqnarray}
\end{enumerate}

We are ready to conclude the following convergence results for \eqref{eq:mapping_nl}, i.e., gradient algorithm with nonlinear power flow as feedback.
\begin{theorem}[Convergence]\label{the:converge}
Given constant stepsize $\epsilon$ chosen to be: 
\begin{eqnarray}
0<\epsilon<2M/L^2, \label{eq:epsilon}
\end{eqnarray}
dynamics~\eqref{eq:mapping_nl} converges as:
\begin{eqnarray}
\lim_{t\rightarrow\infty}\sup\|\bm{z}^{t+1}-\bm{z}^{t+1*}\|^2_2 = \frac{\Delta_1+\epsilon^2\Delta_2}{2\epsilon M-\epsilon^2L^2}. \label{eq:converge}
\end{eqnarray}
\end{theorem}
\begin{proof}
We characterize the 2-norm of the distance between the system states $\bm{z}$ and the optimal value $\bm{z}^*$ at time $t+1$ as follows:
\begin{eqnarray}
&&\|\bm{z}^{t+1}-\bm{z}^{t+1*}\|^2_2\nonumber\\
&=&\|\bm{z}^{t+1}-\bm{z}^{t*}-(\bm{z}^{t+1*}-\bm{z}^{t*})\|^2_2\nonumber\\
&=&\|[\bm{z}^t-\epsilon \tilde{\bm{f}}^t(\bm{z}^t)]_{\Omega}-[\bm{z}^{t*}-\epsilon \bm{f}^t(\bm{z}^{t*})]_{\Omega}\nonumber\\
&&\hspace{2mm}-(\bm{z}^{t+1*}-\bm{z}^{t*})\|^2_2\nonumber\\
&\leq&\|\bm{z}^t-\epsilon \tilde{\bm{f}}(\bm{z}^t)-\bm{z}^{t*}+\epsilon \bm{f}^t(\bm{z}^{t*})\|_2^2+\|\bm{z}^{t+1*}-\bm{z}^{t*}\|^2_2\nonumber\\
&\leq&\|\bm{z}^t-\epsilon \bm{f}^t(\bm{z}^t)-\bm{z}^{t*}+\epsilon \bm{f}^t(\bm{z}^{t*})\|_2^2+\|\bm{z}^{t+1*}-\bm{z}^{t*}\|^2_2\nonumber\\
&&\hspace{2mm}+\epsilon^2\|\tilde{\bm{f}}^t(\bm{z}^t)-\bm{f}^t(\bm{z}^t)\|^2_2\nonumber\\
&\leq&\|\bm{z}^t-\epsilon \bm{f}^t(\bm{z}^t)-\bm{z}^{t*}+\epsilon \bm{f}^t(\bm{z}^{t*})\|_2^2+\Delta_1+\epsilon^2\Delta_2\nonumber\\
&=&\|\bm{z}^t-\bm{z}^{t*}\|_2^2+\|\epsilon \bm{f}^t(\bm{z}^t)-\epsilon \bm{f}^t(\bm{z}^{t*})\|_2^2\nonumber\\
&&\hspace{2mm}-2\epsilon(\bm{z}^t-\bm{z}^{t*})^{\top}(\bm{f}^t(\bm{z}^t)-\bm{f}^t(\bm{z}^{t*})+\Delta_1+\epsilon^2\Delta_2\nonumber\\
&\leq&(1+\epsilon^2L^2-2\epsilon M)\|\bm{z}^t-\bm{z}^{t*}\|_2^2+\Delta_1+\epsilon^2\Delta_2\nonumber\\
&\leq& (1+\epsilon^2L^2-2\epsilon M)^t\|\bm{z}^1-\bm{z}^{1*}\|_2^2\nonumber\\
&&\hspace{2mm}+(\Delta_1+\epsilon^2\Delta_2)\frac{1-(1+\epsilon^2L^2-2\epsilon M)^t}{2\epsilon M-\epsilon^2L^2}\nonumber
\end{eqnarray}
where the first inequality comes from the nonexpansiveness of the projection operation and triangular inequality of the norm, the third from \eqref{eq:delta1}--\eqref{eq:delta2}, the forth from \eqref{eq:B1}--\eqref{eq:B2}, and the last is obtained by repeating previous steps for $t$ times.

Under \eqref{eq:B3}--\eqref{eq:epsilon}, we have $0<1+\epsilon^2L^2-2\epsilon M<1$. Therefore, \eqref{eq:converge} follows.
\end{proof}

Theorem~\ref{the:converge} shows that the online gradient algorithm with nonlinear power flow as feedback tracks the time-varying optimal estimation results within bounded distance. Note that when $\Delta_1 = \Delta_2 = 0$, \eqref{eq:mapping_nl} is equivalent to dynamics \eqref{eq:grad}, and when $\Delta_1 = 0$, \eqref{eq:mapping_nl} is equivalent to dynamics \eqref{eq:grad} with voltage feedback from nonlinear power flow. We immediately come to the following convergence corollary for dynamic \eqref{eq:grad}.
\begin{corollary}\label{the:feedback}
	Given \eqref{eq:epsilon}, dynamics \eqref{eq:grad} converges as:
	\begin{eqnarray}
	\lim_{s\rightarrow\infty}\|\bm{z}(s)-\bm{z}^{*}\|^2_2 = 0, \label{eq:converge_c1}
	\end{eqnarray} 
	and dynamics \eqref{eq:grad} with voltage feedback from nonlinear power flow converges as
	\begin{eqnarray}
	\lim_{s\rightarrow\infty}\sup\|\bm{z}(s)-\bm{z}^{*}\|^2_2 = \frac{\Delta_2}{2M/\epsilon-L^2}. \label{eq:converge_c2}
	\end{eqnarray}
\end{corollary}

The result of \eqref{eq:converge_c2} also indicates that one can achieve an arbitrarily small distance from $\bm{z}$ to $\bm{z}^*$ by appropriately choosing stepsize $\epsilon$. For details, see Section~IV-C in \cite{zhou2019accelerated} to avoid repetition.

\end{document}